\documentclass[12pt]{article}
\usepackage{amsfonts,amsmath,latexsym,amssymb,mathrsfs,amsthm, url,comment}
\usepackage{slashbox}
\usepackage{caption}

\let\OLDthebibliography\thebibliography
\renewcommand\thebibliography[1]{
  \OLDthebibliography{#1}
  \setlength{\parskip}{3pt}
  \setlength{\itemsep}{0pt plus 0.3ex}
}

\def\numberlikeadb{\global\def\theequation{\thesection.\arabic{equation}}}
\numberlikeadb
\newtheorem{theorem}{Theorem}[section]

\newtheorem{corollary}[theorem]{Corollary}
\newtheorem{definition}[theorem]{Definition}

\newtheorem{remark}[theorem]{Remark}
\newtheorem{example}[theorem]{Example}

\usepackage{color}

\definecolor{orange}{rgb}{1,0.5,0}

\usepackage{lscape}
\usepackage{caption}
\usepackage{multirow}
\begin{document}

\title{Compound Poisson approximation of subgraph counts in stochastic block models with multiple edges}
\author{Matthew Coulson\footnote{School of Mathematics, University of Birmingham, Edgbaston, Birmingham B15 2TT, UK},\, Robert E. Gaunt\footnote{School of Mathematics, The University of Manchester, Manchester M13 9PL, UK}\, and
 Gesine Reinert\footnote{Department of Statistics, University of Oxford, 24-29 St Giles', Oxford OX1 3LB, UK}  
}

\date{\today} 
\maketitle

\vspace{-10mm}

\begin{abstract}We use the Stein-Chen method to obtain compound Poisson approximations for the distribution of the number of subgraphs in a generalised stochastic block model which are isomorphic to some fixed graph.  This model generalises the classical stochastic block model to allow for the possibility of multiple edges between vertices.  Both the cases that the fixed graph is a simple graph and that it has multiple edges are treated.  The former results apply when the fixed graph is a member of the class of strictly balanced graphs and the latter results apply to a suitable generalisation of this class to graphs with multiple edges.  We also consider a further generalisation of the model to pseudo-graphs, which may include self-loops as well as multiple edges, and establish a parameter regime in the multiple edge stochastic block model in which Poisson approximations are valid.  The results are applied to obtain Poisson and compound Poisson approximations  (in different regimes) for subgraph counts in the Poisson stochastic block model and degree corrected stochastic block model of Karrer and Newman \cite{kn11}.
\end{abstract}


\noindent{{{\bf{AMS 2010 Subject Classification:}}} Primary 90B15, 62E17, 60F05, 05C80

\section{Introduction}\label{intro}

Small subgraph counts can be used as summary statistics for large random graphs; indeed in some graph models they appear as sufficient statistics, see \cite{frank}. Furthermore, many networks are conjectured to have characteristic over- or under-represented motifs (small subgraphs), see for example \cite{alon}.  Moreover, statistics based on small subgraph counts can be used to compare networks, as in  \cite{sarajlic13, wegner}.   To determine which small subgraphs are unusual, assessing the distribution of such motifs is key.  The mean and variance of subgraph counts for some common random graph models are  given by \cite{picard}, whilst the focus of this paper concerns distributional approximation of subgraph counts.  

A powerful and popular tool for deriving distributional approximations for subgraph counts is the Stein-Chen method \cite{stein, chen 0}.  The method was first used to derive approximations for subgraph counts in \cite{b82} and a comprehensive account of Poisson approximation of subgraph counts is given in the book \cite{bhj92}.  To date, distributional approximations for subgraph counts have mostly been derived for simple random graph models (simple graphs are graphs with no self-loops or multiple edges).  However, networks with multiple edges (often referred to as multigraphs) arise naturally in many real-world networks and also in important synthetic networks models, such as the configuration model. For example in collaboration networks as analysed in  \cite{girvannewman}, vertices are authors, and two authors are linked by an edge if they have co-authored at least one paper.  One can argue that edges in collaboration networks simplify relationships and that the number of joint papers contains important information. Another example is the Molloy-Reed construction of realisations from the configuration model \cite{molloyreed} which allows self-loops as well as multiple edges. A collection of currently 90 networks with multi-edges can be found in the KONECT data base, \url{http://konect.uni-koblenz.de/}. 
 In this paper, we derive distributional approximations for subgraph counts in a generalised stochastic block model, which includes the Poisson stochastic block model and the degree-corrected stochastic blockmodel of \cite{kn11} as  special cases.  

The stochastic block model (SBM) was introduced originally for directed graphs by \cite{holland} and generalised to other graphs by   \cite{ns01}.  It is also called the Erd\H{o}s-R\'{e}nyi Mixture Model  in \cite{daubin} and in theoretical computer science it is known as the Planted Partition Model \cite{condonkarp}.  It has  a wide range of applications; see the survey \cite{matias} and references therein.  In this paper, we shall use the following generalisation of the classical SBM to one with multiple edges.  Consider an undirected random graph on $n$ vertices, with no self-loops, in which the vertices are spread among $Q$ hidden classes. Letting   $f= (f_1,\ldots,f_Q)$ be a vector in $[0,1]^Q$ such that $\sum_{q=1}^Q f_q =1$,  the class label of a vertex is drawn from a multinomial distribution ${\cal{M}}(1, f)$, and class assignments are independent of each other. The number of edges $\tilde{Y}_{i,j}$ between vertices $i$ and $j$ (the tilde in the notation distinguishes these random variable from the random variables $Y_{a,b}$ defined in (\ref{yab}))  are independent conditionally on the class of the vertices, and the probability depends only on the classes of the vertices: denoting by $c_v \in \{1, \ldots, Q\}$ the class of vertex $i$,
\begin{equation}\label{hom}\mathbb{P}(\tilde{Y}_{i,j}=k\,|\,c_i= a, c_j = b)=\pi_{a,b,k}, \quad k=0,1,2,\ldots,
\end{equation}
and $\sum_{k=0}^\infty\pi_{a,b,k}=1$ for all $a$ and $b$.  We shall refer to this model as a \emph{stochastic block multigraph model} (SBMM) and denote it by $SBMM(n,\pi,f)$.  The model of \cite{kn11} corresponds to the case of Poisson probabilities $\pi_{a,b,k}=\mathrm{e}^{-\omega_{a,b}}\omega_{a,b}^k/k!$.  If $\pi_{a,b,k}=0$ for $k\geq2$, then the SBMM model reduces to the classical SBM.   Further, if $\pi_{a,b,0}=1-p, \pi_{a,b,1}=p, \pi_{a,b,k}=0$ for all $k \geq 2$, the SBM reduces to the Erd\H{o}s-R\'{e}nyi random graph model, which we denote by $\mathscr{G}(n,p)$. In Section \ref{sec5}, we shall also introduce a generalisation of the SBMM to pseudo-graphs (graphs which  allow for self-loops as well as multiple edges). In this paper, it is assumed that $\pi$ and $f$ are known,  and that  $f_1,\ldots,f_Q>0$; for estimating these quantities for the classical SBM see, for example, \cite{airoldi, latoucherobin, 
matias, olhedewolfe}.  

Most of this paper shall focus on the case that vertices in the same class are stochastically equivalent, that is the probabilities are given by (\ref{hom}).  However, as noted by \cite{kn11}, such models may be a poor fit when modelling networks with hubs or highly varying degrees within communities.   Indeed, \cite{kn11} introduced the degree-corrected stochastic block model in which the probabilities are given by $\pi_{a,b,i,j,k}=\mathrm{e}^{-\theta_i\theta_j\omega_{a,b}}(\theta_i\theta_j\omega_{a,b})^k/k!$, where $\theta_i$ relates to an expected degree of vertex $i$.  All our results and proofs generalise easily to such heterogeneous vertices, in which the right-hand side of (\ref{hom}) is of the form $\pi_{a,b,i,j,k}$ (see Corollary \ref{cor88} and Example \ref{ex9}).    

In \cite{cgr16}, it is shown that the distribution of the number copies of a fixed graph $G$ in the $SBM(n,\pi,f)$ model is well approximated by an appropriate Poisson distribution if $G$ is a member of the class of strictly balanced graphs (defined in Section \ref{sec2.1}) as long as $\pi_{a,b,1}$ is not too large for all $a$ and $b$.  Explicit bounds for the rate of convergence (in the total variation distance) are also given in  \cite{cgr16}.  The results of \cite{cgr16} generalise well-known results for the $\mathscr{G}(n,p)$ model, such as Theorem 5.B in  \cite{bhj92}.  For normal approximations of subgraph counts there are results available for the $\mathscr{G}(n,p)$ model as well as for the configuration model, see \cite{BR17} and references therein. 

This paper generalises the results of \cite{cgr16} to derive compound Poisson approximations for subgraph counts in the $SBMM(n,\pi,f)$ model.  We consider both the case that the fixed graph $G$ is a simple graph and that it is itself a multigraph. For simple graphs, a Poisson approximation is valid when the fixed graph $G$ is strictly balanced, see for example Theorem 5.B in \cite{bhj92}. For the case of mulitgraphs we introduce a related notion which we call ``strictly pseudo-balanced'', see Definition \ref{pseudo} below. The existence of multiple edges can create clumps of subgraph counts. In such a situation a compound Poisson approximation is more appropriate than a Poisson approximation; see also  Remark \ref{trirem} for why this is the case. The compound Poisson approximations of the main results (Theorems \ref{thm2.1} and \ref{thm3.1}) are valid when the fixed graph $G$ is strictly balanced (or strictly pseudo-balanced) and the edge probabilities $\pi_{a,b,k}$ are not too large for all $a$, $b$ and $k$, and thus generalise Theorem 2.1 of \cite{cgr16} in a natural manner.  In 
certain parameter regimes, which includes the model of \cite{kn11}, not only a compound Poisson approximation but also a Poisson approximation may hold; see Section \ref{sec5}.  As far as we are aware, this is the first paper to give explicit bounds for a compound Poisson approximation of subgraph counts in random multigraph models. 


The rest of the paper is organised as follows. Section \ref{sec2} introduces the setup we use to prove our main results.  Section \ref{sec1.1} contains the notation and procedure we use to count subgraph copies in the SBMM, while Section \ref{sec2.3} recalls the relevant theory from the Stein-Chen method for compound Poisson approximation that will be used to derive the approximations.  Section \ref{sec2.4} concerns the calculation of the parameter in the limiting compound Poisson distribution. As an example, the parameter is worked out in the case that the fixed graph $G$ is the complete graph on $v$ vertices.  In Section \ref{sec3}, we use the Stein-Chen method to derive a compound Poisson approximation for the number of subgraphs in the SBMM which are isomorphic to some fixed simple graph from the class of strictly balanced graphs.   Remark \ref{trirem} illustrates with an example  that a Poisson approximation is in general not applicable.  It is noted in Example \ref{ex3.4} that Theorem \ref{thm2.1} is easily applicable to the model of \cite{kn11}.  Also, in Corollary \ref{cor88} and Example \ref{ex9}, we show how Theorem \ref{thm2.1} generalises to edge probabilities of the form $\pi_{a,b,i,j,k}$, which includes the degree-corrected stochastic block model of \cite{kn11} as a special case. Section \ref{sec4} is devoted to a generalisation of this problem in which the fixed graph can now itself have multiple edges.  We again derive a compound Poisson approximation, for which the strictly balanced graph condition is replaced by an appropriate generalisation for multigraphs. Section \ref{sec5} contains the generalisation of our results to the SBMM with self-loops. We obtain some simple sufficient conditions for which a Poisson approximation of subgraph counts holds in the SBMM (see Theorem \ref{thm5.2} and Remark \ref{rem5.4}).  Finally, Corollary \ref{cor5.5},  gives a Poisson approximation for subgraph counts in the Poisson stochastic block model of \cite{kn11}.  

\section{The setup}\label{sec2}

\subsection{Subgraph counts in the SBMM}\label{sec1.1}

\subsubsection{Simple graph $G$}\label{sec2.1}

Let us now introduce some notation that will be useful for counting copies of a fixed graph $G$ in the $SBMM(n,\pi,f)$.  For simplicity of exposition, we firstly consider the case that $G$ is a simple graph and then in Section \ref{sec2.2} we generalise to the case that $G$ is a multigraph.  Let $K_n$ be the complete graph with $n$ vertices and $\binom{n}{2}$ edges.  Let $G= (V(G), E(G)) \subset K_n$ be a fixed graph with $v(G)$ vertices and $e(G)$ edges; here $V(G)$ denotes its vertex set and $E(G)$ its edge set.  To avoid trivialities, we assume that $e(G)\geq 1$ and that $G$ has no isolated vertices. We shall be particularly interested in the case that $G$ is a member of the class of strictly balanced graphs, which is defined as follows (see for example \cite{bhj92}).  Let
\begin{equation} \label{dg}
d(G) = \frac{e(G)}{v(G)}.
\end{equation}
Then the graph $G$ is said to be \emph{strictly balanced} if $d(H)<d(G)$ for all subgraphs $H\subsetneq G$. 

Let $\Gamma$ be the set of all copies of $G$ in $K_n$, that is all subgraphs of $K_n$ that are isomorphic to $G$.  Note that $|\Gamma|=\binom{n}{v(G)}\rho(G)$, where $\binom{n}{v(G)}$ is the number of ways of choosing $v(G)$ vertices, and  
\begin{equation}\label{rho}\rho(G)=\frac{(v(G))!}{a(G)},
\end{equation}
where $a(G)$ is the number of elements in the automorphism group of $G$.

In a multigraph, more than one isomorphic copy of $G$ may be present for a given labelled graph $\alpha =(V(\alpha), E(\alpha)) \in\Gamma$.  To take this into account, for $\alpha \in \Gamma$ with  $E(\alpha) = \{ (\alpha_1, \alpha_2, \ldots, \alpha_{e(G)})  \}$ let $\Lambda_\alpha= \Lambda_\alpha(G)$ denote the set $\{(a_{\alpha_1},\ldots,a_{\alpha_{e(G)}})\in\{1,2,3,\ldots\}^{e(G)} \} = \mathbb{N}^{e(G)}$.  We view the $a_{\alpha_i}$ as edge labels in the following sense.  In the $SBMM(n,\pi,f)$ model, the number of edges between vertices $u$ and $v$ 
is determined according to the classes of these vertices and the measure $\pi$.  If $k\geq1$ edges are selected via the measure $\pi$, then these edges are given edge labels $1,\ldots,k$.  In this way, if vertices $u$ and $v$ are from classes $a$ and $b$, then edge number $1$ between $u$ and $v$ occurs with probability$\sum_{k=1}^\infty\pi_{a,b,k}$, edge number 2 between $u$ and $v$  occurs with probability $\sum_{k=2}^\infty\pi_{a,b,k}$, and so forth.

Now, let $\mathscr{G}=(\mathcal{V},\mathcal{E})$ be a random multigraph on $n$ edges. For $\alpha\in\Gamma$ and $a\in\Lambda_\alpha$,  let $X_{\alpha,a}(G)$ be the indicator random variable for the occurrence of $\alpha\in\Gamma$ which is isomorphic to the fixed graph $G$, using edges $a\in \Lambda_\alpha$. Here we write as shorthand ``using edges $a\in \Lambda_\alpha$'' to indicate that the occurrence is on edge labelled $a_{\alpha_1} $ of edge $\alpha_1$, on  edge labelled $a_{\alpha_2} $ of edge $\alpha_2$, and so on, until edge labelled $a_{\alpha_{e(G)}} $ of edge $\alpha_{e(G)}$.  Let $W(G)=W$ denote the total number of copies of $G$ in the random graph $\mathscr{G}$,
\begin{equation}\label{weqn}W(G)  = W=\sum_{\alpha\in\Gamma}\sum_{a\in\Lambda_\alpha}X_{\alpha,a} (G).
\end{equation}
In the following, the dependence of $W$ and $X_{\alpha,a} (G)$ on $G$ is usually suppressed to simplify notation.
Here, copies are counted as opposed to induced copies where not only all edges of the graph have to appear, but also no edge which is not in the graph is allowed to appear in the copy. For example, the complete graph $K_n$, $n\geq3$,  contains $(n-1)!/2$ copies, but no induced copy, of an $n$-cycle. 

As an illustration of this notation, count the number of isomorphic copies of the path on three vertices, denoted by $G$, in a graph $\mathscr{G}$ with vertex set $\{1,2,3\}$.  Suppose that $\mathscr{G}$ has a single edge between vertices 1 and 2, three edges between vertices 2 and 3, and no edge between vertices 1 and 3.  Then the set of possible copies of $G$ in $K_3$ is $\Gamma =\{ ( \{1,2,3\}, \{(1,2), (2,3)\} ), ( \{1,2,3\}, \{(1,3), (2,3)\} ), ( \{1,2,3\}, \{(1,2), (1,3)\} ) \} = \{\alpha^{(1)}, \alpha^{(2)}, \alpha^{(3)}\}$. In this example $X_{\alpha^{(1)},(1,1)}=X_{\alpha^{(1)},(1,2)}=X_{\alpha^{(1)},(1,3)}=1$ and all other indicators are equal to 0.  Thus there are three copies of $G$ in $\mathscr{G}$, and hence $W(G) = 3$.

In the $SBMM(n,\pi,f)$, the conditional occurrence probability of an isomorphic copy of the subgraph $G$ on $\alpha = ( \{ {i_1,\ldots,i_{v(G)} \} , \{ {\alpha_1},\ldots,\alpha_{e(G)}} \}  )\in\Gamma$ with edge multiplicities $(a_{\alpha_1}, \ldots, a_{\alpha_{e(G)}}) \in \mathbb{N}^{e(G)}$,  given the classes $c_{i_1}, \ldots, c_{i_{v(G)}}$  of the vertices $i_1,\ldots,i_{v(G)}$ forming $\alpha$ is
\begin{equation*}
\mathbb{P}(X_{\alpha,a}=1\,| c_{i_1}, \ldots, c_{i_{v(G)}} )=\prod_{1\leq u< v\leq v(G): (u,v) \in E(\alpha) }q_{c_u,c_v,a_{(u,v)}},
\end{equation*}  
where $q_{a,b,k}=\sum_{l=k}^\infty \pi_{a,b,l}$ is the probability of there being at least $k$ edges between vertices of classes $a$ and $b$.  Here $a_{(u,v)}$ denotes the edge label $a_{(u,v)}$ between vertices $u$ and $v$.  The occurrence probability of an isomorphic copy of $G$, using edges $a\in \Lambda_\alpha$ is then 
\begin{equation*}\label{mueqn0}
\mu_{\alpha,a}(G)=\mathbb{E} X_{\alpha,a}= \sum_{c_1,c_2,\ldots,c_{v(G)}=1}^{Q} f_{c_1}f_{c_2}
\cdots f_{c_{v(G)}} \prod_{1 \leq u<v \leq v(G): (u,v) \in E(\alpha) }q_{c_u,c_v,a_{(u,v)}},
\end{equation*}
and the expected number of isomorphic copies of $G$ at position $\alpha\in\Gamma$ is therefore
\begin{align}
\label{mugggl}\mu(G)&=\sum_{a\in\Lambda_\alpha}\mathbb{E} X_{\alpha,a}\\
&= \sum_{a\in\Lambda_\alpha}\sum_{c_1,c_2,\ldots,c_{v(G)}=1}^{Q} f_{c_1}f_{c_2}
\cdots f_{c_{v(G)}} \prod_{1 \leq u<v \leq v(G): (u,v) \in E(\alpha) }q_{c_u,c_v,a_{(u,v)}}\nonumber\\
&=\sum_{c_1,c_2,\ldots,c_{v(G)}=1}^{Q} f_{c_1}f_{c_2}
\cdots f_{c_{v(G)}} \prod_{1 \leq u<v \leq v(G): (u,v) \in E(\alpha) }\bigg(\sum_{a_{(u,v)}=1}^\infty q_{c_u,c_v,a_{(u,v)}}\bigg)\nonumber\\
\label{mueqn}&=\sum_{c_1,c_2,\ldots,c_{v(G)}=1}^{Q} f_{c_1}f_{c_2}
\cdots f_{c_{v(G)}} \prod_{1 \leq u<v \leq v(G): (u,v) \in E(\alpha) }\mathbb{E}Y_{c_u,c_v},
\end{align}
where the random variable $Y_{a,b}$ has probability mass function 
\begin{equation}\label{yab}\mathbb{P}(Y_{a,b}=k)=\pi_{a,b,k}, \quad k=0,1,2,\ldots.
\end{equation}
We write $\mu(G)$ (with no $\alpha$ subscript) in (\ref{mugggl}), because $\mu(G)$ is constant over $\alpha\in\Gamma$, for a fixed
$\Gamma$, since all the graphs in $\Gamma$ are copies of the same graph.
From \eqref{weqn} and \eqref{yab} it follows that  
\begin{equation}\label{lambdaeqn}\mathbb{E}W=\binom{n}{v(G)} \rho(G) \mu(G).
\end{equation}

\subsubsection{Multigraph $G$}\label{sec2.2}

Let us now generalise the setup for simple fixed graphs $G$ to fixed multigraphs $G$.  Assume that the graph $G$ has $v(G)$ vertices and denote the maximum number of edges between any two nodes by the finite number $t(G)$.  As before, the vertex and edge sets are denoted by $V(G)$ and $E(G)$ respectively.  Here $E(G)$ differs from the simple graph case, because there may now be more than one edge between a pair of vertices.  Multiple edges appear in $E(G)$ according to their multiplicities in $G$.  Again, assume that there are no isolated vertices and the total number of edges $e(G)$ is  greater or equal to 1.  The notion of strictly balanced graphs has the following natural generalisation to multigraphs.  

\begin{definition}\label{pseudo}Let $f(G)$ be the number of pairs of vertices in $G$ with at least one edge between them and set
\begin{equation}\label{partial}
\partial(G) = \frac{f(G)}{v(G)}.
\end{equation}
Then we say that $G$ is a member  of the class of \emph{strictly pseudo-balanced multigraphs} if $\partial(H)<\partial(G)$ for all subgraphs $H\subsetneq G$. 
\end{definition}

Note that the condition \eqref{partial} is equivalent to requiring the subgraph obtained by reducing the multiplicity of all edges to 1 to be strictly balanced.

The procedure for counting copies of $G$ is similar to the simple graph case, but with some important generalisations. The quantity $\rho(G)$ is defined as in (\ref{rho}).  Though it is worth noting, for example, that if $G_1$ is a triangle then $\rho(G_1)=1$, whereas for the multigraph $G_2$ which is a triangle with an additional edge added between a pair of vertices then $\rho(G_2)=3$.  Note that there is a discrepancy here even though the cardinality of the automorphism groups would be equal if the multiplicity of the multi-edge in $G_2$ was set to 1.  Now, $\Gamma$ is the set of all isomorphic copies of $G$ in the complete $t(G)$-multigraph on $n$ vertices, for which there are} $t(G)$ edges between all $\binom{n}{2}$ vertex pairs.  As before, $|\Gamma|=\rho(G)\binom{n}{v(G)}$.  Note, though that for counting copies of $G_1$ we have $|\Gamma|=\binom{n}{3}$, but for counting copies of $G_2$ the set $\Gamma$ has a larger cardinality: $|\Gamma|=3\binom{n}{3}$.  


We are also required to suitably generalise the set $\Lambda_\alpha$. To explain the generalisation, first suppose that $G$ consists of only one edge, with multiplicity $k$. Then $\alpha \in \Gamma$  consists of $k$ edges
between a particular pair of vertices.  If the observed count is $j\ge k$, then we consider all possible ${j \choose k}$ possibilities to pick $k$ edges out of the $j$ edges.  In this case, to reflect the choices for all possible $j \ge k$ we set $\Lambda_\alpha=:\Lambda_{\alpha}^k= \{ (a_1, \ldots, a_k) \in \mathbb{N}^k: 1 \le a_1 < a_2 < \cdots <a_k \}$, where the superscript emphasises that the edge in $G$ has multiplicity $k$.  If $G$ has $e(G)$ multi-edges instead of just one multi-edge, with edge multiplicities $k_1,\ldots,k_{e(G)}$, and if $E(\alpha) = \{\alpha_1, \ldots, \alpha_{e(G)}\}$,  then we let
$$\Lambda_\alpha = \Lambda_{\alpha_1}^{k_1} \otimes \Lambda_{\alpha_2}^{k_2} \otimes \cdots \otimes \Lambda_{\alpha_e(G)}^{k_{e(G)}} = \bigotimes_{i=1}^{e(G)} \Lambda_{\alpha_i}^{k_i}. 
$$

  
      
        As an example, consider the multigraph $\mathcal{G}$ on three vertices in which there are two edges between each vertex pair.  Then there are eight copies of the triangle graph $G_1$, but twelve copies of the multigraph $G_2$.  Since $\rho(G_2)$=3, there are three isomorphic copies of $G_2$ in $\Gamma$, which we denote by $\alpha_1,\alpha_2,\alpha_3$.  The indicator random variables are $X_{\alpha_i,(j)\otimes(k)\otimes(1,2)}=1$, $i=1,2,3$, $j=1,2$ and $k=1,2$.  The total number of copies of the multigraph $G$ is then defined as in (\ref{weqn}).

Recall that in Section \ref{sec2.1} we defined the random variable $Y_{a,b}$ to have probability mass function $\mathbb{P}(Y_{a,b}=k)=\pi_{a,b,k}$, $k=0,1,2,\ldots$.  Let $i_{(u,v)}\geq1$ be the number of edges between vertices $u$ and $v$ in the fixed multigraph $G$.  Then, by generalising the argument used to obtain (\ref{mueqn}), we have that
\begin{align*}
\mu(G)=\sum_{a\in\Lambda_\alpha}\mathbb{E} X_{\alpha,a}&=\sum_{c_1,c_2,\ldots,c_{v(G)}=1}^{Q} f_{c_1}f_{c_2}
\cdots f_{c_{v(G)}}\times \\
&\quad\times \prod_{1 \leq u<v \leq v(G): (u,v) \in E(G) }\bigg(\sum_{a_{(u,v)}=i_{(u,v)}}^\infty \binom{a_{(u,v)}-1}{i_{(u,v)}-1}\mathbb{P}(Y_{c_u,c_v}\geq a_{(u,v)})\bigg).
\end{align*}
Here, to avoid double counting, we assume the edges in the SBMM appear in some order, so picking the $a_{(u,v)}$-th as well as $i_{(u,v)}-1$ smaller edges leads to the binomial coefficient in the above sum.
Now, 
\begin{align*}\sum_{a=i}^\infty \binom{a-1}{i-1}\mathbb{P}(Y_{c_u,c_v}\geq a)&=\sum_{a=i}^\infty \binom{a-1}{i-1}\sum_{\ell=a}^\infty \mathbb{P}(Y_{c_u,c_v}= \ell)\\
&=\sum_{\ell=i}^\infty\sum_{a=i}^\ell \binom{a-1}{i-1}\mathbb{P}(Y_{c_u,c_v}= \ell)\\
&=\sum_{\ell=i}^\infty \binom{\ell}{i} \mathbb{P}(Y_{c_u,c_v}= \ell)=\mathbb{E}\bigg[\binom{Y_{c_u,c_v}}{i}\bigg],
\end{align*}
and so 
\begin{equation}\label{mueqn2}\mu(G)=\sum_{a\in\Lambda_\alpha}\mathbb{E} X_{\alpha,a}=\sum_{c_1,c_2,\ldots,c_{v(G)}=1}^{Q} f_{c_1}f_{c_2}
\cdots f_{c_{v(G)}} \prod_{1 \leq u<v \leq v(G): (u,v) \in E(G) }\mathbb{E}\bigg[\binom{Y_{c_u,c_v}}{i_{(u,v)}}\bigg].
\end{equation}
(Note that when $G$ is a simple graph we have $i_{(u,v)}=1$, and so (\ref{mueqn2}) reduces to (\ref{mueqn}).)  The computation of $\mathbb{E}W$ is then as before, with (\ref{mueqn2}) now replacing (\ref{mueqn}). {Lastly, we note that in the case of one class ($Q=1$) formula (\ref{mueqn2}) simplifies.  Suppose that $G$ has $e_i(G)$ pairs of vertices which have $i$ edges between them, with $e_{t(G)}(G)>0$ and $e_i(G)=0$ for all $i>t(G)$.  Then
\begin{equation*}\mu(G)=\prod_{i=1}^{t(G)}\bigg(\mathbb{E}\bigg[\binom{Y_{1,1}}{i}\bigg]\bigg)^{e_i(G)}.
\end{equation*}}

\begin{remark}\label{trirem}In \cite{cgr16}, a Poisson approximation was obtained for the distribution of the number of isomorphic copies of a strictly balanced graph in the SBM.  The possibility of multiple edges in the SBMM model means that in general a Poisson approximation will not be valid.  This can be seen from the following simple example.  Consider the $SBMM(n,f,\pi)$ with edge probability distribution $\pi_{a,b,0}=1-p$, $\pi_{a,b,2}=p$ for all $a,b$, and $\pi_{a,b,k}=0$ for $k=1$ and $k\geq3$ and all $a,b$.  This model can be seen as the classical Erd\H{o}s-R\'{e}nyi random graph model, with two edges occurring (rather than just one) between two vertices independently and uniformly with probability $p$.  Triangles (a member of class of strictly balanced graphs) occur in clumps of 8 in this model, and the distribution of the total number of triangles is well-approximated by a $8Po(\nu)$ distribution, where $\nu=\binom{n}{3}p^3$ is the expected number of triangles (this can be easily deduced from Theorem 5.B of \cite{bhj92}).  Clearly, $8Po(\nu)$ is not Poisson distributed.  Instead it is a special case of a compound Poisson distribution. A compound Poisson distribution is the distribution of the random sum $\sum_{n=1}^N X_i$, where the ``number of clumps'' $N$ follows a Poisson distribution and the ``clump sizes'' $X_i, i=1, 2, \ldots $ are i.i.d$.$ (and can be constant, for example all equal to 8) and independent of $N$.  Hence in this paper we focus on compound Poisson approximations. 
\end{remark}

\subsection{The Stein-Chen method for compound Poisson approximation}\label{sec2.3}

In this paper, we use the Stein-Chen method for compound Poisson approximation to assess the distributional distance between $\mathcal{L}(W)$ and the compound Poisson  $CP(\boldsymbol{\lambda})$ distribution when the fixed graph $G$ is a member of the class of strictly balanced graphs or strictly pseudo-balanced graphs.   In this section, we consider the case that $G$ is a simple graph or multigraph in one framework.  This distributional distance is measured using the total variation distance, which for non-negative, integer-valued random variables $U$ and $V$ is given by
\begin{equation*}d_{TV}(\mathcal{L}(U),\mathcal{L}(V))=\sup_{A\subseteq\mathbb{Z}^+}|\mathbb{P}(U\in A)-\mathbb{P}(V\in A)|.
\end{equation*} 
We begin by recalling the compound Poisson distribution, and then provide the relevant details of the Stein-Chen method for compound Poisson approximation.

The compound Poisson distributions $CP(\boldsymbol{\lambda})$ is a family of distributions with an infinite-dimensional parameter $\boldsymbol{\lambda}=(\lambda_1,\lambda_2,\ldots)$ such that $\lambda_i\geq0$ for all $i=1,2,\ldots$. We suppose that $\lambda:=\sum_{i=1}^\infty\lambda_i<\infty$. Then a random variable having the $CP(\boldsymbol{\lambda})$ distribution can be constructed as follows. Let $X_1,X_2,\ldots,$ be a sequence of independent and identically distributed random variables with probability mass function $\mathbb{P}(X_1=i)=\lambda_i/\lambda$. Also, let $Z$ be a Poisson distributed random variable independent of the $\{X_j\}$, with mean $\mathbb{E}Z=\lambda$. With these definitions of $\{X_j\}$ and $Z$, the random sum $\sum_{j=1}^Z X_j$ has the $CP(\boldsymbol{\lambda})$ distribution. The $CP(\boldsymbol{\lambda})$ distribution can be expressed in a second way. If $Z_i$ are independent Poisson variables with means $\mathbb{E}Z_i=\lambda_i$, then $\sum_{i=1}^\infty iZ_i$ is $CP(\boldsymbol{\lambda})$ distributed. In general the probability mass function of a compound Poisson distribution is not available in closed form.

The Stein-Chen method for Poisson approximation is introduced by \cite{chen 0}, and is extended to compound Poisson approximation by \cite{r93,r94}.  A comprehensive account of the application of the Stein-Chen method for Poisson approximation in random graph theory is given in \cite{bhj92}, and the method is used to derive compound Poisson approximation of subgraph counts in Erd\H{o}s-R\'{e}nyi random graphs in \cite{stark}.  Here, we present the compound Poisson framework that is given in \cite{r93,r94} and \cite{stark}.

The random variable of interest is assumed to be of the form $W=\sum_{\alpha\in\Gamma}\sum_{a\in\Lambda_\alpha}X_{\alpha,a}$.  The set $\Gamma\setminus\{\alpha\}$ is partitioned into three index classes for each $\alpha\in\Gamma$.  These index classes are denoted by $\Gamma_\alpha^s$, $\Gamma_\alpha^b$, and $\Gamma_\alpha^w$, so that $\Gamma\setminus\{\alpha\}=\Gamma_\alpha^s\cup \Gamma_\alpha^b\cup \Gamma_\alpha^w$.  The intuition behind this partitioning is as follows. Think of $W$ as a sum of random elements which are dependent, but the dependence is fairly local, with only a weak long-range dependence, if any at all. In this situation one tries to partition the index set such that the 
 set $\Gamma_\alpha^s$ roughly consists of indices $\beta$ such that the indicators $\{X_{\beta,b}\,:\,\beta\in \Gamma_\alpha^s\}$ are strongly dependent on $X_{\alpha,a}$ (for any $a\in\Lambda_a$ and $b\in\Lambda_\beta)$; the set $\Gamma_\alpha^w$ roughly consists of indices $\beta$ such that the indicators $\{X_{\beta,b}\,:\,\beta\in\Gamma_\alpha^w\}$ are very weakly dependent on $X_{\alpha,a}$; and the set $\Gamma_\alpha^b$ (boundary indices) roughly consists of indices $\beta$ such that the indicators $\{X_{\beta,b}\,:\,\beta\in\Gamma_\alpha^b\}$ are perhaps more than weakly dependent on $X_{\alpha,a}$. A natural choice for the sets $\Gamma_\alpha^s$, $\Gamma_\alpha^w$ and $\Gamma_\alpha^b$ in our situation is 
\begin{align*}\Gamma_\alpha^s&=\{\beta\in\Gamma\setminus\{\alpha\}\,:\,V(\alpha)=V(\beta)\}, \\
\Gamma_\alpha^b&=\{\beta\in\Gamma\,:\,|V(\alpha)\cap V(\beta)|\in\{1,2,\ldots,v(G)-1\}\}, \\
\Gamma_\alpha^w&=\{\beta\in\Gamma\,:\,|V(\alpha)\cap V(\beta)|=0\}.
\end{align*}
For example, suppose  that $G$ is the path on three vertices and we consider the complete graph $K_6$ on 6 vertices.  Ignoring redundant copies, on $K_6$ we have the  ${6 \choose 3} \times 3 = 60$ copies of $G$ in $K_6$, such as  $\{(1,2),(2,3)\}$,
$\{(1,3),(2,3)\}$ and $\{(1,2),(1,3)\}$. Suppose $\alpha=\{ \{1,2,3\},\{(1,2),(2,3)\}\}$.  Then with our definition, $\Gamma_\alpha^s=\{ \{(1,3),(2,3)\}, \{(1,2),(1,3)\}\}$, while 
$\Gamma_\alpha^w = \{ \{( 4,5) , ( 4, 6 )\} , \{(4,5) , (5,6)\},  \{( 4,6) , ( 5,6)\}\}$, and $\Gamma_\alpha^b = \Gamma \setminus ( \Gamma_ \alpha^{s} \cup \Gamma_\alpha^w \cup \{ \alpha\})$.

The set $\{\beta\in\Gamma_\alpha^s, \,b\in\Lambda_\beta\,:\,X_{\beta,b}=1\}$ is the ``clump" about $\alpha$. The size of the clump at $\alpha$ (which includes $\alpha$ itself) is denoted by $Z_\alpha$ and given by the equation
\begin{equation}\label{zeqn}Z_\alpha=\sum_{a\in\Lambda_\alpha}X_{\alpha,a}+\sum_{\beta\in\Gamma_\alpha^s}\sum_{b\in\Lambda_\beta}X_{\beta,b}.
\end{equation}
With our choice of index sets, for multigraph counts  $\{X_{\beta,b}\,:\,\beta\in\Gamma_\alpha^w\}$ is independent of $X_{\alpha,a}$.

The parameter $\boldsymbol{\lambda}=(\lambda_1,\lambda_2,\ldots)$ of the approximating compound Poisson distribution is given by
\begin{equation}\label{lamlam}\lambda_i=\frac{1}{i}\sum_{\alpha\in\Gamma}\sum_{a\in\Lambda_a}\mathbb{E}[X_{\alpha,a} I(Z_\alpha=i)], \quad i=1,2,\ldots.
\end{equation}
(This is slight variant of the expression given in \cite{r93,r94}, which takes into account the presence of multiple edges.)  The $\lambda_i$  represent the expected number of ``clumps'' of size $i$. They are calculated explicitly for many examples in \cite{r93}, although none of these formulas apply to random graphs.  Further details on the computation of the $\lambda_i$ for subgraph counts in the SBMM are given in Section \ref{sec2.4}.

Before stating the Stein-Chen bound for compound Poisson approximation that will be used in this paper, we introduce some further notation.  Define
\begin{align*}\eta_\alpha=\sum_{\beta\in\Gamma\setminus\Gamma_\alpha^w}\sum_{b\in\Lambda_\beta}X_{\beta,b}, \quad V_\alpha=\sum_{\beta\in\Gamma_\alpha^b}\sum_{b\in\Lambda_\beta}X_{\beta,b},
\end{align*}
and
\begin{equation*}\epsilon=\sum_{\alpha\in\Gamma}\sum_{a\in\Lambda_\alpha}\big\{\mathbb{E}X_{\alpha,a}\mathbb{E}\eta_\alpha+\mathbb{E}[X_{\alpha,a}V_\alpha]\big\}.
\end{equation*}
Now we can state the Stein-Chen bound for Possion approximation (adapted from \cite{r94}, Theorem 2) that will be used in this paper:
\begin{equation}\label{dtv}d_{TV}(\mathcal{L}(W),CP(\boldsymbol{\lambda})) \leq c(\boldsymbol{\lambda})\epsilon.
\end{equation}
Here
\begin{align*}c(\boldsymbol{\lambda})&=\sup_{A\subset \mathbb{Z}^+}\sup_{j\geq0}|g_{\boldsymbol{\lambda},A}(j+1)-g_{\boldsymbol{\lambda},A}(j)|,
\end{align*}
where $\mathbb{Z}^+=\{0,1,2,\ldots\}$ and $g_{\boldsymbol{\lambda},A}(j)$ is the unique bounded solution of the so-called Stein equation
\begin{equation*}\sum_{i=1}^\infty i\lambda_i g(j+i)-jg(j)=I(j\in A)-\mathbb{P}(Y\in A),
\end{equation*}
for $Y\sim CP(\boldsymbol{\lambda})$.  It should be noted that our notation for $\eta_\alpha$, which involves summing over $\beta\in\Gamma\setminus\Gamma_\alpha^w$ and thus includes $\alpha$, is different from the analogous notation in \cite{r94}, but is convenient in the subgraph counts applications considered in this paper.

There exist bounds in the literature for the constant $c(\boldsymbol{\lambda})$.  It is shown in Theorem 4 of \cite{bcl92} that 
\begin{equation}\label{c2eqn}c(\boldsymbol{\lambda})\leq \mathrm{e}^{\lambda} \min\{1,\lambda_1^{-1}\},
\end{equation}
for all $\boldsymbol{\lambda}$ with $\lambda = \sum_{i=1}^\infty\lambda_i<\infty$.  In general, the dependence on $\boldsymbol{\lambda}$ in such  a bound on  $c(\boldsymbol{\lambda})$ cannot be improved, and the bound \eqref{c2eqn} is  most useful when $\lambda$ is small. However, in certain settings, the estimate (\ref{c2eqn}) can be improved; see \cite{bcl92, bc01, bu98, bx00, d17}.

\begin{remark}
Theorem 2 of \cite{r94} contains an additional term which depends on 
\begin{equation*}\phi_{\alpha,a, i}=\mathbb{E}|\mathbb{E}[X_{\alpha,a} I(Z_\alpha=i)\,|\,(X_{\beta,b}\,:\,\beta\in\Gamma_\alpha^w,b\in\Lambda_\beta)]-\mathbb{E}[X_{\alpha,a} I(Z_\alpha=i)]|.
\end{equation*}
By independence of vertex-disjoint edges, and the fact that the edges in $\Gamma_{\alpha}^s$ are vertex-disjoint from those in $\Gamma_{\alpha}^w$, we have    that $Z_\alpha$ is independent of the random variables $\{X_{\beta,b}\,:\,\beta\in\Gamma_\alpha^w,b\in\Lambda_\beta\}
$.  Therefore $\phi_{\alpha,a, i}=0$ for all values of $\alpha$, $a$ and $i$.  This means that this addtional term vanishes from the bound of Theorem 2 of \cite{r94}. 
\end{remark} 

In summary, bounding the total variation distance between the distribution of the subgraph counts in the SBMM and the $CP(\boldsymbol{\lambda})$ distribution reduces to bounding the quantity $\epsilon$.  We shall derive our compound Poisson approximations for subgraph counts (Theorems \ref{thm2.1} and \ref{thm3.1}) using this approach.

\subsection{Calculating the parameter of the limiting compound Poisson distribution}\label{sec2.4}

The majority of the approximation theorems derived in this paper will involve a compound Poisson approximation distribution of the subgraph counts in the SBMM, with the exception being the Poisson approximations of Section \ref{possec}.  In the case of Poisson approximation, the parameter $\nu:=\mathbb{E}W$ of the limiting $Po(\nu)$ distribution is easily calculated using (\ref{lambdaeqn}).  For example in the Erd\H{o}s-R\'{e}nyi $\mathscr{G}(n,p)$ model, which is a special case of the SBMM, if the fixed graph $G$ is a triangle, then $v(G)=3$, $\rho(G)=1$ and $\mu(G)=p^3$, and so $\nu=\binom{n}{3}p^3$.  In Section \ref{possec}, it shall be understood that $\nu=\mathbb{E}W$ is calculated via (\ref{lambdaeqn}).  

When a compound Poisson approximation is sought, the parameter $\boldsymbol{\lambda}=(\lambda_1,\lambda_2,\ldots)$ is given by (\ref{lamlam}).  Given any fixed $\alpha^*\in\Gamma$, by symmetry we can write (\ref{lamlam}) as
\begin{align}\lambda_i&=\frac{1}{i}\sum_{\alpha\in\Gamma}\sum_{a\in\Lambda_\alpha}\mathbb{E}X_{\alpha,a}\mathbb{P}(Z_\alpha=i\,|\,X_{\alpha,a}=1)\nonumber \\
\label{conprob}&=\frac{1}{i}|\Gamma|\sum_{a\in\Lambda_{\alpha^*}}\mathbb{E}X_{\alpha^*,a}\mathbb{P}(Z_{\alpha^*}=i\,|\,X_{\alpha^*,a}=1) \\
\label{conprob2} &=\frac{1}{i}|\Gamma|\sum_{a\in\Lambda_{\alpha^*}}\mathbb{P}(Z_{\alpha^*}=i\, ,\,X_{\alpha^*,a}=1).
\end{align}
In the simple graph case ($\pi_{a,b,i}=0$ for $i\geq2$ and all $a,b$), a further slight simplification of expression (\ref{conprob}) is possible; see expression (20) of \cite{stark}. 

For fixed graphs $G$ in the SBMM, the conditional probability in (\ref{conprob}) or the  joint probability in (\ref{conprob2}) can be worked out on a case-by-case basis with the formulation (\ref{conprob}) more suitable in a coupling context and (\ref{conprob2}) perhaps more suitable when local dependence is employed.  The conditional probability in (\ref{conprob}) is worked out by \cite{bu98} in the $\mathscr{G}(n,p)$ setting for the Whisker graph (the graph on four vertices comprised of a vertex that is connected by an edge to one of the vertices of the complete graph on three vertices - it should be noted, though, that the Whisker graph is not strictly balanced and \cite{bu98} used a different choice for the sets $\Gamma_\alpha^s$, $\Gamma_\alpha^b$, $\Gamma_\alpha^w$.)    However, as noted by \cite{stark}, it is not always possible to calculate the conditional probability exactly. 

Let us now illustrate the computation of the $\lambda_i$ via expression (\ref{conprob2}) for subgraphs in the SBMM.  For purposes of exposition, we consider the $SBMM(n,\pi,f)$ model with one class, that is $f_1=1$ and $f_i=0$ for $i\geq2$; the extension to multiple classes is not difficult, but involves more notation.  Let $p_i=\pi_{1,1,i}$.

\begin{example}Let us calculate $\boldsymbol{\lambda}$ for the case that $G$ is the complete graph on $v$ vertices with $t$ edges between each vertex pair (a strictly pseudo-balanced graph). We have $|\Gamma|=\binom{n}{v}$, and for $\alpha^*\in\Gamma$, we have
\begin{align*}\Lambda_{\alpha^*}=\{&(a_{1,1},a_{2,1},\ldots,a_{t,1})\otimes\cdots\otimes(a_{1,\binom{v}{2}},a_{2,\binom{v}{2}},\ldots,a_{t,\binom{v}{2}})\in\mathbb{N}^{t\binom{v}{2}}\, :\,  \\
& \text{$1\leq a_{1,1}<a_{2,1}<\ldots< a_{t,1},\ldots, 1\leq a_{1,\binom{v}{2}}<a_{2,\binom{v}{2}}<\ldots< a_{t,\binom{v}{2}}$}\}.
\end{align*}
For $\alpha^* \in \Gamma$ we denote $a \in \Lambda_{\alpha^*}$ by $a=(a_{1,1},\ldots,a_{t,\binom{v}{2}})$. Also, $\Gamma_{\alpha^*}^s=\emptyset$ and so $Z_{\alpha^*}=\sum_{a\in\Lambda_{\alpha^*}}X_{\alpha^*,a}$.  As $G$ is a complete graph, clump sizes take the form of products of binomial coefficients only, with one binomial coefficient per multi-edge.  The joint probability is now given by
\begin{align*}&\mathbb{P}(Z_{\alpha^*}=i\, ,\,X_{\alpha^*,a}=1)=\sum_{x_1,\ldots,x_{\binom{v}{2}}\in B_{a,v,t,i}}p_{x_1}\cdots p_{x_{\binom{v}{2}}}\\
&\quad\quad\quad\text{for all $a=(a_{1,1},\ldots,a_{t,1},\ldots,a_{1,\binom{v}{2}},\ldots, a_{t,\binom{v}{2}})\in A_{v,t,i}$},
\end{align*}
and is 0 otherwise.  Here
\begin{align*}
A_{v,t,i}&=\Big\{a_{1,1},\ldots,a_{t,\binom{v}{2}}\, :\, 1\leq a_{1,1}<\ldots< a_{t,1},\ldots, 1\leq a_{1,\binom{v}{2}}<\ldots< a_{t,\binom{v}{2}} \\
& \quad\quad \text{and } 1 \leq \prod_{\ell=1}^{\binom{v}{2}} \binom{a_{t,\ell}}{t}\leq i \Big\}, \\
B_{a,v,t,i}&= \Big\{x_1,\ldots,x_{\binom{v}{2}}\, :\, \text{$x_1\geq a_{t,1},\ldots,x_{\binom{v}{2}}\geq a_{t,\binom{v}{2}}$ and $\prod_{\ell=1}
^{\binom{v}{2}} \binom{x_\ell}{t}= i$}\Big\}.
\end{align*}
Substituting into (\ref{conprob2}) then gives
\begin{equation}\label{lamex}\lambda_{i}=\frac{1}{i}\binom{n}{v}\sum_{a_{1,1},\ldots,a_{t,\binom{v}{2}}\in A_{v,t,i}}\sum_{x_1,\ldots,x_{\binom{v}{2}}\in B_{a,v,t,i}}p_{x_1}\cdots p_{x_{\binom{v}{2}}}, \quad i\geq 1.
\end{equation}

Note that in the case $t=1$ and  $p_i=0$ for $i\geq2$ the SBMM reduces to the $\mathscr{G}(n,p_1)$ model and we have $\lambda_1=\binom{n}{v}p_1^{\binom{v}{2}}$ and $\lambda_i=0$ for $i\geq2$.  Thus, the limit distribution is simply the Poisson distribution $Po(\binom{n}{v}p_1^{\binom{v}{2}})$, a result which is well known (see Theorem 5.B of \cite{bhj92}). 

For the edge graph ($v=2$), formula (\ref{lamex}) takes a particularly simple form:
\begin{align*}\lambda_{\binom{i}{t}}=\frac{1}{\binom{i}{t}}\binom{n}{2}\sum_{a_{1,1},\ldots,a_{t,1}\in A_{2,t,\binom{i}{t}}}p_i=\frac{1}{\binom{i}{t}}\binom{n}{2}p_i\sum_{1\leq a_{1,1}<a_{2,1}<\ldots<a_{t,1}\leq i}1=\binom{n}{2}p_i, \quad i\geq t,
\end{align*}
and $\lambda_k=0$ for all other $k$.

Finally, we note that $\boldsymbol{\lambda}$ for the edge graph but now with $Q$ classes is easily seen to be
\begin{equation*}\lambda_{\binom{i}{t}}=\binom{n}{2}\sum_{c_1,c_2=1}^Qf_{c_1}f_{c_2}\pi_{c_1,c_2,i}, \quad i\geq t,
\end{equation*}
with $\lambda_k=0$ for all other $k$. Formula (\ref{lamex}) generalises to $Q$ classes similarly.
\end{example}

\section{Compound Poisson approximation in the SBMM: the case of simple fixed graphs}\label{sec3}
In this section, we obtain a compound Poisson approximation for the number of subgraphs in the SBMM which are isomorphic to a fixed graph from the class of strictly balanced graphs.  Before stating this result, we introduce some notation.  Let
\begin{equation}\label{alpha}
\alpha(G) = \min_H \frac{e(G)-e(H)}{v(G)-v(H)}
\end{equation}
and, with \eqref{dg},
\begin{equation}\label{gamma} 
\gamma(G) =\min_H \{ d(G)v(H)-e(H) \} = \min_H [v(H) \{ d(G)-d(H) \} ],
\end{equation}
where the minima are taken over all non-empty proper subgraphs $H \subsetneq G$. In interpreting the following theorem (see, for example, Remark \ref{rem1}), it is useful to note that if $\gamma(G)>0$ or $\alpha(G)>d(G)$ then the graph $G$ is strictly balanced; see \cite{bhj92}.  Finally, we denote
\begin{equation}\label{pi**}\mu_1^*=\max_{a,b}\mathbb{E}Y_{a,b}
\end{equation}
where $Y_{a,b}$ is given in \eqref{yab}.
With this notation, we can state our theorem.

\begin{theorem}\label{thm2.1} Suppose that $G$ is a strictly balanced graph.  Then, with the  notation \eqref{rho}, \eqref{mueqn}, \eqref{lamlam}, \eqref{alpha}, \eqref{gamma} and \eqref{pi**},
\begin{align}\label{ermg1111thm}d_{TV}(\mathcal{L}(W),CP(\boldsymbol{\lambda}))  \leq M_{n,\pi,f}(G)&:=  \frac{c(\boldsymbol{\lambda})\rho(G)^2}{v(G)!}n^{v(G)}(\mu_1^*)^{e(G)} \bigg\{ \frac{v(G)^2}{v(G)!} n^{v(G)-1}   (\mu_1^*)^{e(G)} 
\nonumber \\
 &\quad+ \sum_{i=1}^{v(G)-1} \binom{v(G)}{i}  \frac{n^{v(G)-i}(\mu_1^*)^{\kappa(G,i)}}{(v(G)-i)!} \bigg\},  
\end{align}
where 
\begin{equation} \label{kappa} 
\kappa(G,i)=\max\{e(G)-id(G)+\gamma(G),(v(G)-i)\alpha(G)\}.
\end{equation} 
A bound for $c(\boldsymbol{\lambda})$ is given by inequality (\ref{c2eqn}).
\end{theorem}

\begin{proof}We establish our bound by bounding $\epsilon=\sum_{\alpha\in\Gamma}\sum_{a\in\Lambda_\alpha}
\big\{\mathbb{E}X_{\alpha,a}\mathbb{E}\eta_\alpha+\mathbb{E}[X_{\alpha,a}V_\alpha]\big\}$.  Combining (\ref{mueqn}) and (\ref{pi**}) gives 
\begin{align}\sum_{a\in\Lambda_\alpha}\mathbb{E}X_{\alpha,a}&=\sum_{c_1,c_2,\ldots,c_{v(G)}=1}^{Q} f_{c_1}f_{c_2}
\cdots f_{c_{v(G)}} \prod_{1 \leq u<v \leq v(G): (u,v) \in E(\alpha) }\mathbb{E}Y_{c_u,c_v}\nonumber \\
&\leq \sum_{c_1,c_2,\ldots,c_{v(G)}=1}^{Q} f_{c_1}f_{c_2}
\cdots f_{c_{v(G)}} (\mu_1^*)^{e(G)}\nonumber \\
\label{yineq}&= (\mu_1^*)^{e(G)}.
\end{align}
Next we bound  $\mathbb{E}\eta_\alpha=\sum_{\beta\in\Gamma\setminus\Gamma_\alpha^w}\sum_{b\in\Lambda_\beta}\mathbb{E}X_{\beta,b}$.  Noting that 
\begin{align*}|\Gamma\setminus\Gamma_\alpha^w|&=|\{\beta\in\Gamma_\alpha^b\,:\,|V(\alpha)\cap V(\beta)| \ge 1\}|\\
&= \rho(G) \sum_{k=0}^{v(G)-1 } \binom{n-v(G)}{k} \binom{v(G)}{v(G)-k}  \\
&\leq \rho(G) v(G)  \sum_{k=0}^{v(G)-1 } \binom{n-v(G)}{n-v(G)-k} \binom{v(G)-1}{k}  \\
&=  \rho(G) v(G)  \binom{(n-v(G))+(v(G)-1)}{v(G)-1} \leq\rho(G) v(G)  \binom{n}{v(G)-1}\\
&\leq \rho (G) \frac{v(G)^2}{v(G)!} n^{v(G)-1} ,
\end{align*}
where we used that $\binom{n}{v(G)-1}\leq \frac{n^{v(G)-1}}{v(G)!}$. Combining this bound on $|\Gamma\setminus\Gamma_\alpha^w|$ with (\ref{yineq}) yields
\begin{align}\label{etaineq}\mathbb{E}\eta_\alpha=\sum_{\beta\in\Gamma\setminus\Gamma_\alpha^w}\sum_{b\in\Lambda_\beta}\mathbb{E}X_{\beta,b}\leq|\Gamma\setminus\Gamma_\alpha^w|(\mu_1^*)^{e(G)}\leq  \rho(G)\frac{v(G)^2}{v(G)!} n^{v(G)-1}   (\mu_1^*)^{e(G)}.
\end{align}
From (\ref{yineq}) and (\ref{etaineq}), 
\begin{align}\label{im1}\sum_{\alpha\in\Gamma}\sum_{a\in\Lambda_\alpha}\mathbb{E}
X_{\alpha,a}\mathbb{E}\eta_\alpha&\leq \rho(G)\binom{n}{v(G)}\cdot (\mu_1^*)^{e(G)}\cdot \rho(G)\frac{v(G)^2}{v(G)!} n^{v(G)-1}  (\mu_1^*)^{e(G)}\nonumber \\
&\leq \rho(G)^2\frac{v(G)^2}{(v(G)!)^2}  n^{2v(G)-1 }(\mu_1^*)^{2e(G)}.
\end{align}

Finally, we bound the quantity $\sum_{\alpha\in\Gamma}\sum_{a\in\Lambda_\alpha}\mathbb{E}[X_{\alpha,a}V_\alpha]$.  To do so, we partition $\Gamma_\alpha^b$ into sets $\{\Gamma_\alpha^{b,i}\}_{1\leq i\leq v(G)-1}$, where $\Gamma_\alpha^{b,i}=\{\beta\,:\,|\alpha\cap\beta|=i\}$.  The cardinality of these sets can be bounded above by
\begin{equation}\label{kappa0}|\Gamma_\alpha^{b,i}|\leq \rho(G)\binom{v(G)}{i} \binom{n}{v(G)-i} \leq \rho(G)\binom{v(G)}{i} \frac{n^{v(G)-i}}{(v(G)-i)!}.
\end{equation}
Let us now bound the quantity $\sum_{a\in\Lambda_\alpha}\sum_{b\in\Lambda_\beta}\mathbb{E}[X_{\alpha,a}X_{\beta,b}]$, where $\beta\in\Gamma_\alpha^{b,i}$ for $1\leq i\leq v(G)-1$.  It was shown in the proof of Theorem 2.1 of \cite{cgr16} that for a strictly balanced graph $G$, isomorphic copies $\alpha$ and $\beta$ that share $i$ vertices have at least $e(G)+\kappa(G,i)$ edges in the union graph $\alpha\cup\beta$.  Therefore, by following the steps used to obtain (\ref{mueqn}) and (\ref{yineq}), we obtain
\begin{equation}\label{kappa1}\sum_{a\in\Lambda_\alpha}\sum_{b\in\Lambda_\beta}\mathbb{E}[X_{\alpha,a}X_{\beta,b}]\leq (\mu_1^*)^{e(G)+\kappa(G,i)} \quad \text{for $\beta\in\Gamma_\alpha^{b,i}$}.
\end{equation}
Combining (\ref{kappa0}) and (\ref{kappa1}) now yields
\begin{align}\sum_{\alpha\in\Gamma}\sum_{a\in\Lambda_\alpha}\mathbb{E}[X_{\alpha,a}V_\alpha]&= \sum_{\alpha\in\Gamma}\sum_{i=1}^{v(G)-1}\sum_{\beta_i\in\Gamma_\alpha^{b,i}}\sum_{a\in\Lambda_\alpha}\sum_{b_i\in\Lambda_{\beta_i}}\mathbb{E}[X_{\alpha,a}X_{\beta_i,b_i}] \nonumber \\
&\leq\rho(G)\binom{n}{v(G)}\sum_{i=1}^{v(G)-1}\rho(G)\binom{v(G)}{i} \frac{n^{v(G)-i}}{(v(G)-i)!}\cdot (\mu_1^*)^{e(G)+\kappa(G,i)}\nonumber \\
\label{im2}&\leq\rho(G)^2(\mu_1^*)^{e(G)}\frac{n^{v(G)}}{v(G)!}\sum_{i=1}^{v(G)-1}\binom{v(G)}{i} \frac{n^{v(G)-i}(\mu_1^*)^{\kappa(G,i)}}{(v(G)-i)!}.
\end{align}
Applying (\ref{dtv}) together with the bounds (\ref{im1}) and (\ref{im2}) then yields (\ref{ermg1111thm}).
\end{proof}

\begin{remark}\label{rem1}The result of Theorem \ref{thm2.1} is perhaps most interesting when the limiting $CP(\boldsymbol{\lambda})$ distribution is non-degenerate in the limit $n\rightarrow\infty$.  For this to be the case, we require the mean $\mathbb{E}W$ to be non-degenerate, that is $\mathbb{E}W$ does not tend to 0 or $\infty$ in the limit $n\rightarrow\infty$.  Suppose that there exist universal constants $c$ and $C$ such that $cn^{-1/d(G)}\leq \mathbb{E}Y_{a,b}\leq Cn^{-1/d(G)}$ for all $a,b$.  Then using the inequality 
 $\frac{m^k}{k^k}\leq\binom{m}{k}\leq\frac{m^k}{k!}$, $1\leq k\leq m$,  and \eqref{lambdaeqn} we obtain 
\begin{equation*}\frac{\rho(G)}{v(G)^{v(G)}}c^{e(G)}\leq\mathbb{E}W\leq\frac{\rho(G)}{v(G)!}C^{e(G)}.
\end{equation*} 
In this parameter regime, 
\begin{equation*}\sum_{i=1}^\infty \lambda_i\leq \sum_{i=1}^\infty i\lambda_i=\mathbb{E}W=O(1).
\end{equation*}
Therefore, from (\ref{c2eqn}), we have that $c(\boldsymbol{\lambda})\leq K(G)$ for some $K(G)>0$ which does not involve $n$.  Applying Theorem \ref{thm2.1}  then yields 
\begin{equation} \label{corpn}
d_{TV}(\mathcal{L}(W),CP(\boldsymbol{\lambda})) \leq  \frac{K(G)\rho(G)^2}{v(G)!}C^{e(G)}\bigg\{ \frac{v(G)^2}{v(G)!}C^{e(G)}n^{-1} +\min(A,B) \bigg\}, 
\end{equation}
where
\begin{eqnarray*}A&=&(1+C^{\alpha(G)})^{v(G)-1}n^{1-\alpha(G)/d(G)};\\
B&=&C^{e(G)+\gamma(G)}(1+C^{-d(G)})^{v(G)-1}n^{-\gamma(G)/d(G)}.
\end{eqnarray*}
\end{remark}

\begin{example}Here we use \eqref{corpn} to obtain compound Poisson approximations for the number of copies of the following fixed graphs with $v\geq3$ vertices in the $SBMM(n,\pi,f)$ model.  We consider the following strictly balanced graphs on $v$ vertices each (the same choice as in Remark 2.3 of \cite{cgr16}):

\begin{enumerate}
\item[$G_{1,v}$] a tree on the $v$ vertices, with $v-1$ edges;

\item[$G_{2,v}$] the cycle graph on the $v$ vertices (with $v$ edges);

\item[$G_{3,v}$] the complete graph on $v$ vertices with one edge removed;

\item[$G_{4,v}$] the complete graph on $v$ vertices. 
\end{enumerate}

To apply \eqref{corpn}, we must compute the quantities $d(G)$, $\alpha(G)$ and $\gamma(G)$ for each graph $G$.  These values were computed in Remark 2.3 of \cite{cgr16} and are presented in Table \ref{dag}.  If for a given graph $G$  there exist universal constants $c$ and $C$ such that $cn^{-1/d(G)}\leq \mathbb{E}Y_{a,b}\leq Cn^{-1/d(G)}$ for all $a,b$ (for which the limit distribution is non-degenerate), then a bound on $d_{TV}(\mathcal{L}(W),CP(\boldsymbol{\lambda}))$ now follows from  \eqref{corpn}.  The appropriate scaling of $\mu_1^*=\max_{a,b}\mathbb{E}Y_{a,b}$ (note that all the $\mathbb{E}Y_{a,b}$ are all of the same order) is also reported in the table, as well as a bound on the rate of convergence in this parameter regime. For this rate of convergence it is assumed that the proportion vector $f= f(n) $ remains  constant as $n \rightarrow \infty$,  and that $G$ does not change with $n$.  The conclusions we draw from Table \ref{dag} are very similar to those drawn in Remark 2.3 of \cite{cgr16}; for example, the bound on the rate of convergence for the tree graph may be considerably larger than the bound on the rate of convergence in the cycle graph.  

We also note that if $\mathbb{E}Y_{a,b}=O(n^{-1})$ for all $a$, $b$ (so that $\mu_1^*=O(n^{-1})$) then this gives rise to graphs with bounded average degree, which are important in network science.  With the scaling $Cn^{-1}$ for $\mu_1^*$, the number of isomorphic copies of $G_{2,v}$, the cycle graph on $v$ vertices, is seen to have an approximation non-degenerate compound Poisson distribution.  With the scaling $Cn^{-v/(v-1)}$ for $\mu_1^*$, the number of copies of the tree graph $G_{1,v}$ also has an approximation non-degenerate Compound Poisson distribution, and the average degree is $O(n^{-1/(v-1)})$; however, the scalings required to ensure that the limiting compound Poisson distribution for the number of copies of $G_{3,v}$ and $G_{4,v}$ is non-degenerate gives rises to graphs with unbounded average degree.  

\begin{table}[ht]
\caption{Values of $d(G)$, $\alpha(G)$ and $\gamma(G)$, and scaling and bounds on the rate of convergence} 
\centering

\begin{tabular}{ |p{0.8cm}||p{1.6cm}|p{1.3cm}|p{1.8cm}|p{2.8cm}|p{3.8cm}|  }
 \hline
  $G$ & $d(G)$ &$\alpha(G)$&$\gamma(G)$& Scaling of $\mu_1^*$ & $d_{TV}(\mathcal{L}(W),CP(\boldsymbol{\lambda}))$\\
  \hline
 $G_{1,v}$   & $\frac{v-1}{v}$    & $1$ & $\frac{1}{v}$ & $Cn^{-v/(v-1)}$ & $O(n^{-1/(v-1)})$ \vspace{1mm} \\
 $G_{2,v}$ & 1 & $\frac{v-1}{v-2}$ & $1$ & $Cn^{-1}$ & $O(n^{-1})$ \vspace{1mm}\\
 $G_{3,v}$   & $\frac{(v+1)(v-2)}{2v}$ & $\frac{v^2-v-4}{2(v-2)}$ &  $\frac{1}{3}$ if $v=3$ & $Cn^{-2v/(v+1)(v-2)}$ & $O(n^{-1/2})$ if $v=3$ \vspace{1mm} \\  
 &&& $1$ if $v\geq4$ && $O(n^{-2/(v-1)})$ if $v\geq4$ \\
 $G_{4,v}$  &   $\frac{v-1}{2}$  & $\frac{v+1}{2}$ & $1$ & $Cn^{-2/(v-1)}$ & $O(n^{-2/(v-1)})$ \vspace{1mm}\\
 \hline
\end{tabular}
\label{dag}
\end{table}

\end{example}

\begin{example}\label{ex3.4}
The Poisson stochastic blockmodel from \cite{kn11}, for which the number of edges between vertices of type $a$ and $b$  follows a Poisson distribution with parameter $\omega_{a,b}$, falls into the framework of Theorem \ref{thm2.1}, with $\mu_1^*=\max_{a,b}\mathbb{E}Y_{a,b}=\max_{a,b}\omega_{a,b}=\omega^*$.  According to Remark \ref{rem1}, the mean of the limiting compound distribution is non-degenerate if $cn^{-1/d(G)}\leq\omega_{a,b}\leq Cn^{-1/d(G)}$ for all $a,b$.  In such a regime, the bound (\ref{corpn}) is applicable.

One can similarly apply Theorem \ref{thm3.1} (below) to the Poisson stochastic blockmodel, although we omit the details.  However, in Corollary \ref{cor5.5} we work out the details for a Poisson approximation in the Poisson stochastic blockmodel.
\end{example}

We end this section by considering  the setting of Theorem \ref{thm2.1}, but with edge probabilities (\ref{hom}) given by the more general 
\begin{equation}\label{inhom}\mathbb{P}(\tilde{Y}_{i,j}=k\,|\,c_i= a, c_j = b)=\pi_{a,b,i,j,k}, \quad k=0,1,2,\ldots.
\end{equation}
In this more general setting, the derivation of a bound on the total variation distance between $\mathcal{L}(W)$ and the limiting $CP(\boldsymbol{\lambda})$ is almost unchanged. There is no change to the combinatorial arguments and the sole change is that one bounds the expected number of edges by the quantity $\max_{a,b,i,j}\mathbb{E}Y_{a,b,i,j}$, where the random variable $Y_{a,b,i,j}$ has probability mass function $\mathbb{P}(Y_{a,b,i,j}=k)=\pi_{a,b,i,j,k}$, $k\geq0$.  In this way, one obtains the following.

\begin{corollary}\label{cor88}With the same setting as Theorem \ref{thm2.1}, but the edge probabilities given by (\ref{inhom}) instead of (\ref{hom}), we have the bound
\begin{equation}\label{inhombound}d_{TV}(\mathcal{L}(W),CP(\boldsymbol{\lambda}))\leq  M_{n,\pi,f}^*(G),
\end{equation}
where $M_{n,\pi,f}^*(G)$ is the bound (\ref{ermg1111thm}) of Theorem \ref{thm2.1} with the quantity  $\mu_1^*=\max_{a,b}Y_{a,b}$ replaced by the quantity $\max_{a,b,i,j}\mathbb{E}Y_{a,b,i,j}$.
\end{corollary}
Here $\boldsymbol{\lambda}$ is given by (\ref{lamlam}), but, in contrast to the setting of Theorem \ref{thm2.1}, we cannot compute $\boldsymbol{\lambda}$ using expression (\ref{conprob2}) because the edge probabilities (\ref{inhom}) depend on the vertices at which the edge is incident.  The results derived in the remainder of the paper generalise to edge probabilities of the form (\ref{inhom}) similarly.

\begin{example}\label{ex9}Recall that the degree-corrected stochastic blockmodel of \cite{kn11} has edge probabilities that follow the $Po(\theta_1\theta_j\omega_{a,b})$ distribution.  This model falls into the framework of Corollary \ref{cor88} and we can apply the bound (\ref{inhombound}) with $\max_{a,b,i,j}\mathbb{E}Y_{a,b,i,j}\leq(\theta^*)^2\omega^*$, where $\theta^*=\max_i \theta_i$ and $\omega^*=\max_{a,b}\omega_{a,b}$.
\end{example}

\section{Compound Poisson approximation in the SBMM: extension to fixed multigraphs}\label{sec4}

In this section, we generalise Theorem \ref{thm2.1} to the case that the fixed graph $G$ is a strictly pseudo-balanced multigraph.  We begin by introducing appropriate generalisations of the notation (\ref{alpha}) and (\ref{gamma}).  With $f(G)$ denoting the number of pairs of vertices with at least one edge between them and $\partial(G)$ defined as in (\ref{partial}), we let
\begin{equation}\label{alpha2}
\alpha_m(G) = \min_H \frac{f(G)-f(H)}{v(G)-v(H)}
\end{equation}
and
\begin{equation}\label{gamma2} 
\gamma_m(G) =\min_H \{ \partial(G)v(H)-e(H) \} = \min_H [v(H) \{ \partial(G)-d(H) \} ],
\end{equation}
where the minima are taken over all non-empty proper subgraphs $H \subsetneq G$ such that $f(H)<f(G)$.  In interpreting the following theorem, it is worth noting that the multigraph graph $G$ is strictly pseudo-balanced if $\gamma_m(G)>0$ or $\alpha_m(G)>\partial(G)$.  Finally, we let
\begin{equation}\label{pi***}\mu_k^*=\max_{a,b}\mathbb{E}[Y_{a,b}^k] \quad \text{and} \quad \mu_k^{**}=\max_{a,b}\mathbb{E}\bigg[\binom{Y_{a,b}}{k}\bigg].
\end{equation}
Let us now state our theorem.


\begin{theorem}\label{thm3.1}Suppose that $G$ is a strictly pseudo-balanced graph  which has $e_i(G)$ pairs of vertices which have $i$ edges between them, with $e_{t(G)}(G)>0$ and $e_i(G)=0$ for all $i>t(G)$.  Then, with the  notation \eqref{rho}, \eqref{mueqn2}, \eqref{lamlam}, \eqref{alpha2}, \eqref{gamma2} and \eqref{pi***},
\begin{align}\label{ermg1111thm2}d_{TV}(\mathcal{L}(W),CP(\boldsymbol{\lambda}))  &\leq  \frac{c(\boldsymbol{\lambda})\rho(G)^2}{v(G)!}n^{v(G)} \bigg\{\frac{v(G)^2}{v(G)!}n^{v(G)-1} \prod_{i=1}^{t(G)} \bigg((\mu_i^{**})^{2e_i(G)}\bigg)
\nonumber \\
 &\quad+ \sum_{i=1}^{v(G)-1} \binom{v(G)}{i}  \frac{n^{v(G)-i}(\psi(G))^{e(G)+\kappa_m(G,i)}}{(v(G)-i)!} \bigg\},  
\end{align}
where 
\begin{equation*} \label{kappa9} 
\kappa_m(G,i)=\max\{e(G)-i\partial(G)+\gamma_m(G),(v(G)-i)\alpha_m(G)\},
\end{equation*} 
and 
\begin{equation}\psi(G)=\max\bigg\{2\mu_{2t(G)}^*,\max_{1\leq j\leq t(G)}(\mu_j^{**})\bigg\}.
\end{equation}
\end{theorem}

\begin{proof}The proof follows along the same lines as that of Theorem \ref{thm2.1}.  Again, we bound $\epsilon$.  Firstly, combining (\ref{mueqn2}) and (\ref{pi***}) yields  
\begin{align}\sum_{a\in\Lambda_\alpha}\mathbb{E}X_{\alpha,a}&=\sum_{c_1,c_2,\ldots,c_{v(G)}=1}^{Q} f_{c_1}f_{c_2}
\cdots f_{c_{v(G)}} \prod_{1 \leq u<v \leq v(G): (u,v) \in E(G) }\mathbb{E}\bigg[\binom{Y_{c_u,c_v}}{i_{(u,v)}}\bigg]\nonumber \\
&\leq \sum_{c_1,c_2,\ldots,c_{v(G)}=1}^{Q} f_{c_1}f_{c_2}
\cdots f_{c_{v(G)}}\prod_{i=1}^{t(G)}\Big((\mu_i^{**})^{e_i(G)}\Big)\nonumber \\
\label{yineq2}&= \prod_{i=1}^{t(G)}\Big((\mu_i^{**})^{e_i(G)}\Big).
\end{align}
We bound $\mathbb{E}\eta_\alpha$ in the same way we did in the proof of Theorem \ref{thm2.1} to obtain
\begin{equation*}\mathbb{E}\eta_\alpha\leq n^{v(G)-1}   \rho(G)\frac{v(G)^2}{v(G)!} \prod_{i=1}^{t(G)}\Big((\mu_i^{**})^{e_i(G)}\Big),
\end{equation*}
and thus
\begin{equation}\label{im12}\sum_{\alpha\in\Gamma}\sum_{a\in\Lambda_\alpha}\mathbb{E}X_{\alpha,a}
\mathbb{E}\eta_\alpha\leq  \rho(G)^2\frac{v(G)^2}{(v(G)!)^2}n^{2v(G)-1}  \prod_{i=1}^{t(G)}\bigg((\mu_i^{**})^{2e_i(G)}\bigg).
\end{equation}

Finally, we bound the quantity $\sum_{\alpha\in\Gamma}\sum_{a\in\Lambda_\alpha}\mathbb{E}[X_{\alpha,a}V_\alpha]$.  We can use the bound (\ref{kappa0}) for the cardinality of the sets $\Gamma_\alpha^{b,i}$, $1\leq i\leq v(G)-1$.  However, we must adapt the argument given in Theorem \ref{thm2.1} in order to bound the quantity 
\begin{equation}\label{quan}\sum_{a\in\Lambda_\alpha}\sum_{b\in\Lambda_\beta}\mathbb{E}[X_{\alpha,a}X_{\beta,b}],
\end{equation}
where $\beta\in\Gamma_\alpha^{b,i}$ for $1\leq i\leq v(G)-1$.  Following the proof of Theorem \ref{thm2.1}, we note that for a strictly pseudo-balanced graph $G$, isomorphic copies $\alpha$ and $\beta$ that share $i$ vertices have at least $e(G)+\kappa_m(G,i)$ edges in the union graph $\alpha\cup\beta$.  In bounding (\ref{quan}), we treat the following two cases separately: edge positions in which at least one vertex from the union graph $\alpha\cup\beta$ is not in $V(\alpha)\cap V(\beta)$, and edge positions that are in the intersection $V(\alpha)\cap V(\beta)$.

We consider firstly the case that at least one vertex from the union graph $\alpha\cup\beta$ is not in $V(\alpha)\cap V(\beta)$.  Consider a pair of vertices $(u,v)$ in which $u$ has class $a$ and $v$ has class $b$.  Suppose that at this edge position  the number of edges in the multigraph $G$ is $j$.  Then by arguing as we did in establishing (\ref{mueqn2}) and (\ref{yineq2}) we can bound the expected number occurrences of $j$ edges at this edge position by $\mu_j^{**}$.

Next, we consider the edges inside of the intersection $V(\alpha) \cap V(\beta)$.  So, if $\alpha$ has $p>0$ edges at $(u,v)$ and $\beta$ has $q>0$ edges at $(u,v)$, then the expected number of edges is given by
\begin{equation}\label{edge22}
\sum_{\substack{k_1<k_2<\ldots<k_p \\ l_1<l_2<\ldots<l_q}} \mathbb{P}(Y>k_p \vee l_q) = \sum_{k=p}^\infty \sum_{j=q}^\infty  \binom{k-1}{p-1} \binom{j-1}{q-1} \mathbb{P}(Y \geq k \vee j),
\end{equation}
where $x\vee y=\max\{x,y\}$ and we have set $Y=Y_{a,b}$ for the remainder of this proof.  The above expression comes from considering the different ways in which we could choose the $p$ edges we include in the subgraph located at $\alpha$ and the $q$ in $\beta$. This gives us a count for the number of ways in which we can choose the motif from the edges present once we have picked the location, as we may have more edges than necessary in one location leading to subgraphs appearing in each location with multiplicity.  Also, we note that in bounding (\ref{edge22}) without loss of generality we can suppose $p\geq q$.
Using this argument, \eqref{edge22}  can be bounded as follows; 
\begin{align*}&\sum_{k=p}^\infty \left\{ \sum_{j=q}^k \binom{k-1}{p-1} \binom{j-1}{q-1} \mathbb{P}(Y \geq k)+ 
\binom{k-1}{p-1} \sum_{j=k+1}^\infty  \binom{j-1}{q-1} \mathbb{P}(Y \geq j)
\right\} \\
\quad\quad\quad&=\sum_{k=p}^\infty\binom{k}{q}  \binom{k-1}{p-1}  \mathbb{P}(Y \geq k)+ 
\sum_{j=p+1}^\infty\sum_{k=p}^{j-1} \binom{k-1}{p-1}  \binom{j-1}{q-1} \mathbb{P}(Y \geq j)\\
\quad\quad\quad&=\sum_{k=p}^\infty\binom{k}{q}  \binom{k-1}{p-1}  \mathbb{P}(Y \geq k)+ 
 \sum_{j=p+1}^\infty \binom{j-1}{p}  \binom{j-1}{q-1} \mathbb{P}(Y \geq j)\\
\quad\quad\quad&\leq\sum_{k=p}^\infty\binom{k}{q}  \binom{k-1}{p-1}  \mathbb{P}(Y \geq k)+ 
 \sum_{j=p+1}^\infty \binom{j}{p}  \binom{j-1}{q-1} \mathbb{P}(Y \geq j), 
\end{align*} 
where we used that $\sum_{k=n}^m \binom{k}{n} = \binom{m+1}{n+1}$.  Now we bound for, $k \ge p \ge q$, 
\begin{align*}
\binom{k}{q} \binom{k-1}{p-1} + 
  \binom{k}{p}  \binom{k-1}{q-1}&=\bigg(\frac{1}{q}+\frac{1}{p}\bigg)\frac{1}{(p-1)!(q-1)!}\frac{k!(k-1)!}{(k-q)!(k-p)!} \\
  &\le\frac{p+q}{p! q!} k^{p+q-1}.
\end{align*} 
Therefore (\ref{edge22}) can be bounded above by
\begin{align*}
&\frac{p+q}{p!q!}\sum_{k=0}^{\infty} k^{p+q-1} \mathbb{P}(Y \geq k) \leq \frac{p+q}{p!q!}\sum_{k=0}^{\infty} k^{p+q} \mathbb{P}(Y = k)\\
& \quad= \frac{p+q}{p!q!} \mathbb{E}[Y^{p+q}] \leq 2\mathbb{E}[Y^{p+q}]\leq2\mathbb{E}[Y^{2t(G)}]\leq2\mu_{2t(G)}^*,
\end{align*}
where the penultimate inequality holds because $\mathbb{E}Y \leq \mathbb{E}[Y^2] \leq \ldots \leq \mathbb{E}[Y^{2t(G)}]$, and we used that $\frac{p+q}{p!q!}\leq2$ for all $p,q\geq1$.  Therefore, for each edge within the intersection, the expected number of edges is bounded above by $2\mu_{2t(G)}^*$.

Putting everything together, we obtain, for $\beta\in\Gamma_\alpha^{b,i}$,
\begin{align}\label{kappa2}\sum_{a\in\Lambda_\alpha}\sum_{b\in\Lambda_\beta}\mathbb{E}[X_{\alpha,a}X_{\beta,b}]&\leq\Big(\max\Big\{2\mu_{2t(G)}^*,\max_{1\leq j\leq t(G)}(\mu_j^{**})\Big\}\Big) ^{e(G)+\kappa(G,i)}\nonumber\\
&\leq (\psi(G))^{e(G)+\kappa(G,i)}.
\end{align}
Combining (\ref{kappa0}) and (\ref{kappa2}) now yields
\begin{align}\sum_{\alpha\in\Gamma}\sum_{a\in\Lambda_\alpha}\mathbb{E}[X_{\alpha,a}V_\alpha]&= \sum_{\alpha\in\Gamma}\sum_{i=1}^{v(G)-1}\sum_{\beta_i\in\Gamma_\alpha^{b,i}}\sum_{a\in\Lambda_\alpha}\sum_{b_i\in\Lambda_{\beta_i}}\mathbb{E}[X_{\alpha,a}X_{\beta_i,b_i}] \nonumber \\
&\leq\rho(G)\binom{n}{v(G)}\sum_{i=1}^{v(G)-1}\rho(G)\binom{v(G)}{i} \frac{n^{v(G)-i}}{(v(G)-i)!}\cdot (\psi(G))^{e(G)+\kappa(G,i)}\nonumber \\
\label{im22}&\leq\rho(G)^2\frac{n^{v(G)}}{v(G)!}\sum_{i=1}^{v(G)-1}\binom{v(G)}{i} \frac{n^{v(G)-i}(\psi(G))^{e(G)+\kappa(G,i)}}{(v(G)-i)!}.
\end{align}
Applying (\ref{dtv}) together with the bounds (\ref{im12}) and (\ref{im22}) then yields (\ref{ermg1111thm2}).
\end{proof}

\begin{remark}The limiting $CP(\boldsymbol{\lambda})$ distribution is non-degenerate in the limit $n\rightarrow\infty$ if there exist universal constants $c$ and $C$ such that $cn^{-1/\partial(G)}\leq \mathbb{E}[Y_{a,b}^k]\leq Cn^{-1/\partial(G)}$ for all $a,b$ and all $k=1,2,\ldots,2t(G)$.  In this case, arguing as we did in Remark \ref{rem1} gives that 
\begin{equation*}\frac{\rho(G)}{v(G)^{v(G)}}c^{e(G)}\leq\mathbb{E}W\leq\frac{\rho(G)}{v(G)!}C^{e(G)},
\end{equation*} 
where $e(G)$ is the total number of edges in the multigraph $G$, and we also have that $c(\boldsymbol{\lambda})=O(1)$. 
There then exists a constant $K(G)$, not involving $n$, such that
\begin{equation*}d_{TV}(\mathcal{L}(W),CP(\boldsymbol{\lambda}))\leq K(G)\max\big\{n^{-1},\min\{n^{1-\alpha_m(G)/\partial(G)},n^{-\gamma_m(G)/\partial(G)}\}\big\}.
\end{equation*}
\end{remark}

\section{Further results}\label{sec5}

\subsection{Extension to include self-loops}

Here we consider the following generalisation of the $SBMM(n,\pi,f)$ which includes multiple self-loops and thus allows for pseudo-graphs.  The self-loop count random variables $S_i$, $i=1,\ldots,n$, are assumed to be independent conditionally on the class of vertices, and the probability only depends on the class of the vertex:
\begin{equation*}\mathbb{P}(\tilde{S}_i=k\,|\,i\in a)=\theta_{a,k}, \quad k\geq0.
\end{equation*}
We also assume that $\tilde{S}_1,\ldots,\tilde{S}_n$ are mutually independent.  We denote this model by $SBMM^*(n,\pi,\theta,f)$, where the $*$ in the notation emphasises the generalisation of the SBMM model to include self-loops.  When $\theta_{a,0}=1$ for all $1\leq a\leq Q$, the model reduces to the SBMM. 

Let $G$ be a fixed graph with possibly multiple edges and self-loops and $v(G)<n$ vertices.  Denote by $V(G)$ and $E(G)$ its vertex and edge sets.  The edge set now will include the additional detail as to whether there are edge(s) from node $i$ to node $i$, that is whether there are self-loop(s) at vertex $i$.  It is straightforward to generalise Theorems \ref{thm2.1} and \ref{thm3.1} to the $SBMM^*(n,\pi,\theta,f)$, because the inclusion of self-loops makes no difference to the combinatorial arguments of the proofs.  We shall only focus on the generalisation of \ref{thm3.1}; the generalisation of \ref{thm2.1} follows from exactly the same argument.

The setup of Section \ref{sec2.2} is mostly unchanged by the generalisation to self-loops, although the expectation (\ref{mueqn2}) generalises to
\begin{align}\label{mueqn3}\mu(G)&=\sum_{a\in\Gamma_\alpha}\mathbb{E} X_{\alpha,a}=\sum_{c_1,c_2,\ldots,c_{v(G)}=1}^{Q} f_{c_1}f_{c_2}
\cdots f_{c_{v(G)}}\times\nonumber \\
&\quad\times \bigg\{\prod_{1\leq w\leq V(G): (w,w)\in E(G)}\mathbb{E}S_{c_w}\bigg\}\bigg\{ \prod_{1 \leq u<v \leq v(G): (u,v) \in E(G) }\mathbb{E}\bigg[\binom{Y_{c_u,c_v}}{i_{(u,v)}}\bigg]\bigg\},
\end{align}
where $(w,w)$ denotes a self-loop at vertex $w$ and the random variables $S_a$, $1\leq a\leq Q$, have probability mass function $\mathbb{P}(S_a=k)=\theta_{a,k}$, $k=0,1,2,\ldots$.  Using the convention that the empty product is set to 1, we recover (\ref{mueqn2}) from (\ref{mueqn3}) under the case of no self-loops.  The definition (\ref{partial}) for strictly pseudo-balanced graphs is applicable to graphs $G$ that have self-loops, and is equivalent to the condition that the subgraph $G'$ formed by removing all self-loops from $G$ is itself a strictly pseudo-balanced graph.  The same is true of the notation $\alpha_m(G)$, $\gamma_m(G)$ and $\kappa_m(G,i)$ of Section \ref{sec4}.  The analogue of (\ref{pi**}) and (\ref{pi***}) for self-loops is 
\begin{equation}\label{theta}\phi^*=\max_{a} \mathbb{E}S_a.
\end{equation}
The proof of the following theorem is exactly the same as that of Theorem \ref{thm3.1}, except that the expectations $\mathbb{E}[X_{\alpha,a}X_{\beta,b}]$ are now alternated to take into account the self-loop probabilities, in the same way that we did for (\ref{mueqn3}).  For $\beta\in\Gamma_\alpha^{b,i}$, we can compute the upper bound
\begin{equation*}\sum_{a\in\Lambda_\alpha}\sum_{b\in\Lambda_\beta}\mathbb{E}[X_{\alpha,a}X_{\beta,b}]\leq(\phi^*)^{2s(G)-i}(\mu_{2t(G)}^*)^{e(G)+\kappa(G,i)}. 
\end{equation*}
 The rest of the proof is unchanged and we obtain the following bound.

\begin{theorem}\label{thm5.1}Suppose that $G$ is a strictly pseudo-balanced graph, which has $e_i(G)$ pairs of vertices which have $i$ edges between them, with $e_{t(G)}(G)>0$ and $e_i=0$ for all $i>t(G)$, and $s(G)\geq0$ self-loops.  Then, with the  notation \eqref{rho}, \eqref{lamlam}, \eqref{mueqn3}, \eqref{alpha2}, \eqref{gamma2}, \eqref{pi***}, (\ref{theta}),
\begin{align*}d_{TV}(\mathcal{L}(W),CP(\boldsymbol{\lambda}))  &\leq  \frac{c(\boldsymbol{\lambda})\rho(G)^2}{v(G)!}n^{v(G)} \bigg\{ \frac{v(G)^2}{v(G)!}n^{v(G)-1}  (\phi^*)^{2s(G)}\prod_{i=1}^{t(G)} \Big((\mu_i^{**})^{2e_i(G)}\Big)
 \\
 &\quad+ \sum_{i=1}^{v(G)-1} \binom{v(G)}{i}  \frac{n^{v(G)-i}(\phi^*)^{2s(G)-i}(\psi(G))^{e(G)+\kappa_m(G,i)}}{(v(G)-i)!} \bigg\}. 
\end{align*}
\end{theorem} 

\subsection{Poisson approximation}\label{possec}

As was seen in Remark \ref{trirem}, in general a Poisson approximation for the distribution of the number of isomorphic copies of a strictly balanced graph is not valid in the $SBMM(n,\pi,f)$ model.  However, as will be seen in this section, in a certain parameter regime a Poisson approximation is valid.  Of course, in the case $\pi_{a,b,k}=0$ for all $1\leq a,b \leq Q$ and $k\geq2$, then the model reduces to the classical SBM and a Poisson approximation is valid in the small $\pi_{a,b,1}$ regime (see \cite{cgr16}).  In Theorem \ref{thm5.2}, we obtain a parameter regime in which multiple edges are possible (that is sufficient conditions on $\pi_{a,b,k}>0$ for some $1\leq a,b\leq Q$ and $k\geq2$ such that a Poisson approximation is valid.)  Before stating the theorem, we present the Stein-Chen bound for Poisson approximation that we will use in the proof.

For $\alpha\in\Gamma$, let
\begin{equation*}
A_\alpha = \{ \beta \in \Gamma \colon \lvert \alpha \cap \beta \rvert \geq 1 \},
\end{equation*} 
so that $A_\alpha=\Gamma_\alpha^b\cup\Gamma_\alpha^s\cup\{\alpha\}$.  Also, let
\begin{equation*}
\label{thetaeqn}U_{\alpha,a} =\sum_{\beta\in A_\alpha}\sum_{b\in\Lambda_\beta} X_{\beta,b}-X_{\alpha,a}.
\end{equation*}
Then a simple corollary of Theorem 1 in \cite{agg89}, or of Theorem 1.A in \cite{bhj92}, is that  
\begin{equation}\label{dtv44}
d_{TV}(\mathcal{L}(W),Po(\nu)) \leq \nu^{-1}(1-\mathrm{e}^{-\nu}) \sum_{\alpha \in \Gamma}\sum_{a\in\Lambda_\alpha}  \big\{\mathbb{E}X_{\alpha,a} \mathbb{E}\eta_{\alpha}+\mathbb{E}[X_{\alpha,a}U_{\alpha,a}]\big\}.
\end{equation}
Here 
\begin{equation}\label{nunu}\nu:=\mathbb{E}W=\binom{n}{v(G)}\rho(G)\mu(G),
\end{equation}
where we used Equation (\ref{lambdaeqn}) to write down the expectation.

It is worth comparing (\ref{dtv44}) with the bound $d_{TV}(\mathcal{L}(W),CP(\boldsymbol{\lambda})) \leq c(\boldsymbol{\lambda})\epsilon$ that we used for compound Poisson approximation in Section \ref{sec2.3}.  The term $\nu^{-1}(1-\mathrm{e}^{-\nu})$ is present because $c(\boldsymbol{\lambda})=c((\nu,0,0,\ldots))\leq\nu^{-1}(1-\mathrm{e}^{-\nu})$ for the $Po(\nu)$ distribution (see \cite{be83}, Lemma 4).  The quantity $U_{\alpha}$ can be written as $U_{\alpha,a}=V_\alpha+T_{\alpha,a}$, where $T_{\alpha,a}=\sum_{\beta\in \Gamma_\alpha^s\cup\{\alpha\}}\sum_{b\in\Lambda_\beta} X_{\beta,b}-X_{\alpha,a}$. 
Therefore (\ref{dtv}) can be written as
\begin{align}\label{dtv445}
d_{TV}(\mathcal{L}(W),Po(\lambda)) &\leq \nu^{-1}(1-\mathrm{e}^{-\nu}) \sum_{\alpha \in \Gamma}\sum_{a\in\Lambda_\alpha}  \big\{\mathbb{E}X_{\alpha,a} \mathbb{E}\eta_{\alpha}+\mathbb{E}[X_{\alpha,a}V_\alpha]\big\}\nonumber\\ 
&\quad +\nu^{-1}(1-\mathrm{e}^{-\nu}) \sum_{\alpha \in \Gamma}\sum_{a\in\Lambda_\alpha}\mathbb{E}[X_{\alpha,a}T_{\alpha,a}].
\end{align}
We have already bounded the first term of (\ref{dtv445}) in Theorem \ref{thm2.1}, so to arrive at a Poisson approximation we must also be able to bound the second term.  

\begin{theorem}\label{thm5.2} Suppose that $G$ is a strictly balanced graph.  Then, with the  notation \eqref{rho}, \eqref{mueqn}, \eqref{alpha}, \eqref{gamma}, \eqref{pi**} and \eqref{nunu},
\begin{align}d_{TV}(\mathcal{L}(W),Po(\nu))  &\leq M_{n,\pi,f}(G)+\nu^{-1}(1-\mathrm{e}^{-\nu})\frac{\rho(G)^2}{v(G)!}n^{v(G)}(\mu_1^*)^{e(G)-1}q_2^*\nonumber \\
&= \nu^{-1}(1-\mathrm{e}^{-\nu})\frac{\rho(G)^2}{v(G)!}n^{v(G)}(\mu_1^*)^{e(G)-1} \bigg\{  \frac{v(G)^2}{v(G)!}n^{v(G)-1} (\mu_1^*)^{e(G)+1} 
\nonumber \\
\label{ermg1111thm799}  &\quad+ q_2^* + \sum_{i=1}^{v(G)-1} \binom{v(G)}{i}  \frac{n^{v(G)-i}(\mu_1^*)^{\kappa(G,i)+1}}{(v(G)-i)!} \bigg\}, 
\end{align}
where $M_{n,\pi,f}(G)$ is the bound (\ref{ermg1111thm}) of Theorem \ref{thm2.1} and $q_2^*=\max_{1\leq a,b \leq Q}\sum_{k=2}^\infty\pi_{a,b,k}$.
\end{theorem}

\begin{proof}We bound  $\sum_{\alpha \in \Gamma}\sum_{a\in\Lambda_\alpha}\mathbb{E}[X_{\alpha,a}T_{\alpha,a}]$, where $T_{\alpha,a}=\sum_{\beta\in \Gamma_\alpha^s\cup\{\alpha\}}\sum_{b\in\Lambda_\beta} X_{\beta,b}-X_{\alpha,a}$.  First, we note that $|\Gamma_\alpha^s\cup\{\alpha\}|=\rho(G)$.  To bound $\mathbb{E}[X_{\alpha,a}T_{\alpha,a}]$, we require bounds on $\sum_{a\in\Lambda_\alpha}\sum_{b\in\Lambda_\beta}\mathbb{E}[X_{\alpha,a}X_{\beta,b}]$, where $|\alpha\cap\beta|=v(G)$ and $(\alpha,a)\not=(\beta,b)$.  For all such $\alpha,\beta,a,b$, the largest value that $\sum_{a\in\Lambda_\alpha}\sum_{b\in\Lambda_\beta}\mathbb{E}[X_{\alpha,a}X_{\beta,b}]$ takes is when $\alpha=\beta$, because this minimises the number of edge positions in the intersection graph $\alpha\cap\beta$.  There are $e(G)$ possible edge locations in the intersection graph, and because $(\alpha,a)\not=(\alpha,b)$ it follows that there must be at least two edges at one edge position and at least one edges at the remaining $e(G)-1$ edges positions.  
Hence it follows that
\begin{equation*}\sum_{a\in\Lambda_\alpha}\sum_{b\in\Lambda_\beta}\mathbb{E}[X_{\alpha,a}X_{\beta,b}]\leq q_2^*(\mu_1^*)^{e(G)-1}
\end{equation*}
for all $\alpha,\beta$ such that $|\alpha\cap\beta|=v(G)$.  Therefore
\begin{equation*}\sum_{\alpha \in \Gamma}\sum_{a\in\Lambda_\alpha}\mathbb{E}[X_{\alpha,a}T_{\alpha,a}]\leq \rho(G)q_2^*(\mu_1^*)^{e(G)-1}.
\end{equation*}
Using this inequality to bound the second term of (\ref{dtv445}) and using the bound (\ref{ermg1111thm}) of Theorem \ref{thm2.1} to bound the first term yields (\ref{ermg1111thm799}), as required.
\end{proof}

\begin{remark}In the classic $SBM(n,\pi,f)$ model, there are no multiple edges (that is $\pi_{a,b,k}=0$ for all $a,b$ if $k\geq2$), and thus $q_2^*=0$.  Furthermore, $\mu_1^*=\max_{1\leq a,b \leq Q}\pi_{a,b,1}=:\pi^*$.  The bound (\ref{ermg1111thm799}) then can be seen to be exactly the same as the bound of Theorem 2.1 of \cite{cgr16}, as one would hope for.  
\end{remark}

\begin{remark}\label{rem5.4}In Remark \ref{rem1}, we showed that if there exist universal constants $c$ and $C$ such that $cn^{-1/d(G)}\leq \mathbb{E}Y_{a,b}\leq Cn^{-1/d(G)}$ for all $x$, then the limiting $Po(\nu)$ distribution is non-degenerate, and the bound $M_{n,\pi,f}(G)$ tends to 0 as $n\rightarrow\infty$.  If we also have that $q_2^*\ll n^{-1/d(G)}$, then $d_{TV}(\mathcal{L}(W),Po(\nu))\rightarrow0$ as $n\rightarrow\infty$.  Given that $\mathbb{E}Y_{a,b}=O(n^{-1/d(G)})$, then the condition $q_2^*\ll n^{-1/d(G)}$ is equivalent to 
\begin{equation}\label{cond1}q_2^*=\max_{1\leq a,b\leq Q}\big(\sum_{l=2}^\infty \pi_{a,b,l}\big)\ll \max_{1\leq a,b\leq Q}\pi_{a,b,1}=O(n^{-1/d(G)}).
\end{equation}

A simple and natural example in which condition (\ref{cond1}) is met is the geometric probabilities $\pi_{a,b,k}=p^k (1-p)$, $k \geq 0$, for all $1\leq a,b\leq Q$.  The condition is met if $p\ll 1$.  The model of \cite{kn11}, which has Poisson probabilities $\pi_{a,b,k}=\mathrm{e}^{-\omega_{a,b}}\omega_{a,b}^k/k!$, $k\geq0$, also satisfies condition (\ref{cond1}) if $\max_{1\leq a,b\leq Q}\omega_{a,b}\ll 1$.  Indeed, we obtain an explicit bound for the Poisson approximation of subgraph count distributions in this model in the following corollary.
\end{remark}

\begin{corollary}\label{cor5.5}Suppose we are in the setting of Theorem \ref{thm5.2} be satisfied and suppose that $\pi_{a,b,k}=\mathrm{e}^{-\omega_{a,b}}\omega_{a,b}^k/k!$, $k\geq0$.  Let $\omega^*=\max_{1\leq a,b\leq Q}\omega_{a,b}$.  Then
\begin{align}d_{TV}(\mathcal{L}(W),Po(\nu))  & \leq\nu^{-1}(1-\mathrm{e}^{-\nu})\frac{\rho(G)^2}{v(G)!}n^{v(G)}(\omega^*)^{e(G)-1} \bigg\{   \frac{v(G)^2}{v(G)!} n^{v(G)-1}    (\omega^*)^{e(G)+1} 
\nonumber \\
\label{ermg1111thm7997}  &\quad+ \frac{1}{2}(\omega^*)^2+ \sum_{i=1}^{v(G)-1} \binom{v(G)}{i}  \frac{n^{v(G)-i}(\omega^*)^{\kappa(G,i)+1}}{(v(G)-i)!} \bigg\}.
\end{align}

Suppose now that there exist universal constants $c$ and $C$ such that $cn^{-1/d(G)}\leq \omega_{a,b}\leq Cn^{-1/d(G)}$ for all $a,b$.  Then 
\begin{equation}\label{ineqqq}\frac{\rho(G)}{v(G)^{v(G)}}c^{e(G)}\leq\nu\leq\frac{\rho(G)}{v(G)!}C^{e(G)},
\end{equation}
and so the limiting $Po(\nu)$ distribution is non-degenerate as $n\rightarrow 0$.  Moreover, there then exists a constant $K(G)$, independent of $n$, such that
\begin{equation}\label{ddtt}d_{TV}(\mathcal{L}(W),Po(\nu))\leq K(G)\max\big\{n^{-1},n^{-1/d(G)},\min\{n^{1-\alpha(G)/d(G)},n^{-\gamma(G)/d(G)}\}\big\}.
\end{equation}
\end{corollary}

\begin{proof}We use the bound (\ref{ermg1111thm799}) of Theorem \ref{thm5.2}, so we are required to bound $\mu_1^*$ and $q_2^*$.  Recall from Example \ref{ex3.4} that $\mu_1^*=\omega^*$.  Also, 
\begin{equation*}q_{a,b,2}=\sum_{k=2}^\infty \mathrm{e}^{-\omega_{a,b}}\frac{\omega_{a,b}^k}{k!}= 1-(1+\omega_{a,b})\mathrm{e}^{-\omega_{a,b}}\leq 1-(1+\omega^*)\mathrm{e}^{-\omega^*}\leq \frac{1}{2}(\omega^*)^2,
\end{equation*}
where the inequalities may be proved by simple calculus.  Therefore $q_2^*\leq \frac{1}{2}(\omega^*)^2$.  Inserting these bounds into (\ref{ermg1111thm799}) yields (\ref{ermg1111thm7997}), as required.

We obtain (\ref{ineqqq}) by applying exactly the same argument as the one used in Remark \ref{rem1}.  The bound (\ref{ddtt}) is also very similar to the total variation distance bound of Remark \ref{rem1}.  Here we have an additional term to consider:
\begin{equation*}\frac{1}{2\nu}(1-\mathrm{e}^{-\nu})\frac{\rho(G)^2}{v(G)!}n^{v(G)}(\omega^*)^{e(G)+1}. 
\end{equation*}
This term is order $n^{-1/d(G)}$, and so we obtain (\ref{ddtt}), which completes the proof.
\end{proof}

\subsection*{Acknowledgements}
MC acknowledges support from the Department of Statistics, University of Oxford, for a Summer Studentship.  RG acknowledges support from EPSRC
grant EP/K032402/1 and is currently supported by a Dame Kathleen Ollerenshaw Research Fellowship.  GR acknowledges support from EPSRC grant EP/K032402/1 as well as from the Alan Turing Institute.  The authors would like to thank the referees for their helpful comments and suggestions.

\end{document}